\documentclass[12pt,reqno]{amsart}

\usepackage[colorlinks=true, pdfstartview=FitV, linkcolor=blue, citecolor=blue, urlcolor=blue]{hyperref}

\usepackage[myheadings]{fullpage}

\usepackage[T1]{fontenc}
\usepackage[english]{babel}

\usepackage{amsmath,amssymb,amsfonts,amsthm,enumitem,mathtools,mathrsfs}
\usepackage[utf8]{inputenc}

\allowdisplaybreaks

\numberwithin{equation}{section}

\DeclareMathOperator*{\essinf}{ess\,inf}
\DeclareMathOperator*{\esssup}{ess\,sup}

\newcommand{\pp}{{p(\cdot)}}
\newcommand{\cpp}{{p'(\cdot)}}
\newcommand{\Lp}{L^p}
\newcommand{\Lpp}{L^{p(\cdot)}}

\newcommand{\rr}{{r(\cdot)}}
\newcommand{\LH}{\text{LH}}
\newcommand{\LLH}{\text{LH}_0}
\newcommand{\GLH}{\text{LH}_\infty}
\newcommand{\Ap}{A_p}
\newcommand{\App}{{A_{p(\cdot)}}}

\newcommand{\D}{\mathcal{D}}

\def\Xint#1{\mathchoice
   {\XXint\displaystyle\textstyle{#1}}%
   {\XXint\textstyle\scriptstyle{#1}}%
   {\XXint\scriptstyle\scriptscriptstyle{#1}}%
   {\XXint\scriptscriptstyle\scriptscriptstyle{#1}}%
   \!\int}
\def\XXint#1#2#3{{\setbox0=\hbox{$#1{#2#3}{\int}$}
     \vcenter{\hbox{$#2#3$}}\kern-.5\wd0}}

\def\avgint{\Xint-}

\DeclarePairedDelimiter\ceil{\lceil}{\rceil}
\DeclarePairedDelimiter\floor{\lfloor}{\rfloor}
\DeclarePairedDelimiter{\Norm}{\lVert}{\rVert}

\newtheorem{theorem}{Theorem}[section]
\newtheorem*{theorem*}{Theorem}
\newtheorem{corollary}[theorem]{Corollary}
\newtheorem{lemma}[theorem]{Lemma}
\newtheorem*{lemma*}{Lemma}
\newtheorem{definition}[theorem]{Definition}
\newtheorem*{definition*}{Definition}

\newtheorem*{example*}{Example}
\newtheorem{proposition}[theorem]{Proposition}
\newtheorem*{proposition*}{Proposition}
\newtheorem{conjecture}[theorem]{Conjecture}
\newtheorem*{conjecture*}{Conjecture}
\newtheorem{remark}[theorem]{Remark}
\newtheorem*{remark*}{Remark}

\begin{document}
\title{Weighted norm inequalities for the maximal operator on $\Lpp$ over spaces of homogeneous type}

\author{David Cruz-Uribe, OFS}
\author{Jeremy Cummings}

\thanks{This paper is based on the masters thesis written by the second author.  The first author is supported by \
  research funds from the Dean of the College of Arts \& Sciences, the
  University of Alabama.  }

\date{July 21, 2020}



\begin{abstract}
Given a space of homogeneous type $(X,\mu,d)$, we prove strong-type weighted norm inequalities for the Hardy-Littlewood maximal operator over the variable exponent Lebesgue spaces $L^\pp$. We prove that the variable Muckenhoupt condition $\App$ is necessary and sufficient for the strong type inequality if $\pp$ satisfies log-H\"older continuity conditions and $1 < p_- \leq p_+ < \infty$. Our results generalize to spaces of homogeneous type the analogous results in Euclidean space proved in \cite{paper}.
\end{abstract}

\maketitle



\section{Introduction}
This paper is concerned with extending established results in the theory of variable exponent Lebesgue spaces to the setting of spaces of homogeneous type. In recent decades---largely as a result of \cite{modern}---interest has arisen over the natural extension of the classical Lebesgue spaces $\Lp$ in which the exponent $p$ is itself a function of the underlying space; see \cite{book, Diening} for extensive discussions on such spaces. In particular, the development of a variable exponent Calder\'{o}n-Zygmund theory has been the subject of much research, especially since Cruz-Uribe, Fiorenza, and Neugebauer \cite{unweighted}, building on the work of Diening \cite{early}, proved that the Hardy-Littlewood maximal operator is bounded on $\Lpp$ for $\pp$ satisfying a continuity condition weaker than H\"older continuity.

Just as with the development of classical Calder\'on-Zygmund theory, many results in the theory of variable exponent spaces have only been proved to hold over $\mathbb{R}^n$ or in metric spaces (see \cite{metric} for the latter). In the 1970s, this restriction was removed by Coifman and Weiss, who in \cite{French} introduced spaces of homogeneous type, which they later developed in \cite{English} as the natural spaces onto which Calder\'{o}n-Zygmund theory could be generalized. A logical step for variable exponent theory, then, is to perform the same generalization for $\Lpp$. Such a program has been underway since the maximal operator was shown to be bounded over $\mathbb{R}^n$. Early results include \cite{historytwo,historythree,Khabazi,historyone}; for a more detailed history, see \cite{Adamowicz}.

Spaces of homogeneous type have a topological structure weaker than metric spaces: namely, that of a quasi-metric space.

\begin{definition} \label{quasi}
Given a set $X$ and a function $d\,:\,X \times X \to [0,\infty)$, we say $(X,d)$ is a \textbf{quasi-metric space} if
\begin{enumerate}
\item $d(x,y)=0$ if and only if $x=y$.
\item $d(x,y)=d(y,x)$ for all $x,y \in X$.
\item There exists a constant $A_0 \geq 1$ for which $d(x,y) \leq A_0(d(x,z)+d(z,y))$ for all $x,y,z \in X$.
\end{enumerate}
\end{definition}

The constant $A_0$ is referred to as the \textit{quasi-metric constant}. Some authors (e.g. \cite{symmetry}) also loosen condition (2) to symmetry up to a constant, $d(x,y) \leq Kd(y,x)$. An important property of quasi-metric spaces is that quasi-metric balls need not be open; however, Mac\'ias and Segovia \cite{equivalent} showed that there is always an equivalent quasi-metric whose balls are all open. Analysis can be done on quasi-metric spaces without additional structure---see \cite{Ahlfors}---but typically measures on quasi-metric measure spaces are taken to be at least doubling.

\begin{definition}
A measure $\mu$ on a space $X$ is said to be \textbf{doubling} if there exists a constant $C_\mu \geq 1$ such that, for any $x \in X$ and $r>0$,
\[ 0 < \mu(B(x,2r)) \leq C_\mu\mu(B(x,r)) < \infty, \]
where $B(x,r) = \{y \in X\,:\,d(x,y) < r\}$ is the quasi-metric ball of radius $r$ centered at $x$. The constant $C_\mu$ is called the \textbf{doubling constant}.
\end{definition}

The assumption that balls have positive, finite measure avoids trivial measures, and also ensures that $\mu$ is $\sigma$-finite. We are now led naturally to the well-known setting of spaces of homogeneous type.

\begin{definition}
A \textbf{Space of Homogeneous Type} is a triple $(X,d,\mu)$ where $X$ is a non-empty set, $d$ is a quasi-metric on $X$, and $\mu$ is a doubling regular measure on the $\sigma$-algebra generated by quasi-metric balls and open sets.
\end{definition}

Hereafter, we will let $(X,d,\mu)$ be a fixed space of homogeneous type, and often denote it simply by $X$. The assumption that $\mu$ is regular is used only to apply the Lebesgue Differentiation Theorem in Section~\ref{necessity}; see \cite{Ahlfors} for the possibility of weakening this hypothesis.

We now introduce some basic notions of the variable exponent spaces $\Lpp(X)$.

\begin{definition}
Define $\mathcal{P}(X)$ to be the set of measurable functions $\pp\,:\,X \to [1,\infty]$. The elements of $\mathcal{P}(X)$ are called \textbf{exponent functions}. Given an exponent function and a set $E \subseteq X$, we define
\[ p_-(E) = \essinf_{x \in E} p(x) \qquad p_+(E) = \esssup_{x \in E} p(x). \]
In particular we denote $p_-(X)=p_-$ and $p_+(X)=p_+$.
\end{definition}

When considering the conjugate exponent function $\cpp$ defined by $p'(x)=\frac{p(x)}{p(x)-1}$ (with the convention that $1/0 = \infty$ and $1/\infty = 0$), to avoid the ambiguity inherent to notation like ``$p'_+$'' we will always write $(p')_+$ to denote the essential supremum of $\cpp$, etc.

\begin{definition} Given an exponent $\pp \in \mathcal{P}(X)$, define
\begin{itemize}
\item $X_\infty = \{x \in X \,:\, p(x)=\infty\}$
\item$X_1 = \{x \in X \,:\, p(x)=1\}$
\item$X_* = \{x \in X \,:\, 1<p(x)<\infty\}$.
\end{itemize}
\end{definition}

Intuitively, given an exponent $\pp \in \mathcal{P}(X)$, we would like to define $\Lpp(X)$ as the collection of all functions on $X$ satisfying
\[ \int_X |f|^{p(x)}\,d\mu < \infty. \]
To properly formulate this, we require an analog to the constant exponent $p$-norm. It is well-known that the following modular function provides such an analog.

\begin{definition}
The space $\Lpp(X)$ is the set of measurable functions $f$ on $X$ for which the modular
\[ \rho_\pp(f) = \int_{X \backslash X_\infty} |f(x)|^{p(x)}d\mu+\Norm{f}_{L^\infty(X_\infty)} \]
satisfies $\rho_\pp(f/\lambda) < \infty$ for some $\lambda > 0$.
\end{definition}

If $p_+ < \infty$, we say $f$ is locally $\pp$-integrable if $\rho_\pp(f\chi_B) < \infty$ for every ball $B \subset X$. Often we write $\Lpp$ for $\Lpp(X)$; similarly, when the exponent is clear from context, we will simply write $\rho_\pp = \rho$. It is shown in \cite{book,Diening} that this modular induces the following Luxembourg norm on $\Lpp$.

\begin{proposition}
The function $\Norm{\cdot}\,:\,\Lpp(X) \to \mathbb{R}$ given by
\[ \Norm{f}_{\Lpp(X)} = \inf\{\lambda>0:\rho(f/\lambda)\leq1\} \]
is a norm on $\Lpp(X)$, which is Banach with respect to $\Norm{\cdot}_{L^\pp(X)}$.
\end{proposition}

When the underlying space is clear from context, we write $\Norm{\cdot}_{\Lpp} = \Norm{\cdot}_\pp$. In the case that $\pp$ is constant, $\pp = p$, it is easy to show that $\Norm{\cdot}_\pp$ reduces to the usual norm in $\Lp$.

For most purposes, the set $\mathcal{P}(X)$ of possible exponent functions is far too broad to prove meaningful results. Indeed, even piecewise-constant exponents lose many of the properties of classical $\Lp$ spaces (see \cite{book}), such as the boundedness of the maximal operator. In the study of variable exponent theory, it has become clear that in many cases a sufficient condition on the exponent is log-H\"older continuity.

\begin{definition}
We say that an exponent $\pp \in \mathcal{P}(X)$ is \textbf{locally Log-H\"{o}lder continuous}, $\pp \in \text{LH}_0$, if there exists a constant $C_0$ such that for any $x,y \in X$ with $d(x,y)<1/2$, 
\[ |p(x)-p(y)| < \frac{-C_0}{\log(d(x,y))}. \]
We say that $\pp$ is \textbf{Log-H\"{o}lder continuous at infinity}, $\pp \in \text{LH}_\infty$, with respect to a base point $x_0 \in X$ if there exist constants $C_\infty$ and $p_\infty$ such that for every $x \in X$,
\[ |p(x)-p_\infty| < \frac{C_\infty}{\log(e+d(x,x_0))}. \]
We call $C_0$ the $\text{\bf{LH}}_\mathbf{0}$ \textbf{constant} of $\pp$ and $C_\infty$ the $\text{\bf{LH}}_\mathbf{\infty}$ \textbf{constant} of $\pp$. If $\pp \in \LH = \LLH \cap \text{LH}_\infty$, we say that $\pp$ is \textbf{globally Log-H\"{o}lder continuous}.
\end{definition}

Note that $\pp \in \LH$ implies that $p_+ < \infty$, a condition that is crucial to most of the results in this paper. Note also that the above definition appears to depend on the choice of base point $x_0$. In fact, such a choice is irrelevant, as shown by the following lemma, which was proved in \cite{Adamowicz}.

\begin{lemma} \label{basepoint}
Choose $x_0, y_0 \in X$. If $\pp \in \GLH$ with respect to $x_0$, then $p \in \GLH$ with respect to $y_0$.
\end{lemma}

Whenever $x_0$ is not chosen explicitly, we assume that $X$ has a fixed, arbitrarily chosen base point $x_0$.

We are interested in weighted norm inequalities on $\Lpp$. For classical Lebesgue spaces, much of the theory of such inequalities is due to Muckenhoupt (see e.g. \cite{Muckenhoupt}). The following definition clarifies some standard notation.

\begin{definition}
A \textbf{weight} is a locally integrable function $w\,:\,X \to [0,\infty]$ with $0 < w(x) < \infty$ almost everywhere. Given a weight $w$, we define its associated measure by $dw(x) =  w(x)\,d\mu(x)$. The weighted average integral of a function $f$ over a set $E \subset X$ with $w(E) > 0$ is denoted
\[ \avgint_E f(x)\,dw = \frac{1}{w(E)}\int_E f(x)\,w(x)\,d\mu. \]
If $w = 1$, we replace $dw$ with $d\mu$.
\end{definition}

We denote by $M$ the uncentered Hardy-Littlewood maximal operator; that is,
\[ Mf(x) = \sup_{B \ni x} \avgint |f(y)|\,d\mu. \]
For classical Lebesgue spaces $\Lp(\mathbb{R}^n)$, Muckenhoupt proved in \cite{Muckenhoupt} that a necessary and sufficient condition for strong-type $(p,p)$ weighted norm inequalities, $p > 1$, is that for every ball $B$,
\[ \avgint_B w(x)\,dx\left(\avgint_B w(x)^{1-p'}\,dx\right)^{p-1} \leq C < \infty. \]
This is the famous Muckenhoupt $\Ap$ condition. In \cite{paper} the $\Ap$ condition is recast into an equivalent form which may be generalized to variable exponent spaces.

\begin{definition} \label{App}
Given an exponent $\pp \in \mathcal{P}(X)$, we say $w \in \App$ if there exists a constant $K$ such that for any ball $B$,
\[ \Norm{w\chi_B}_\pp\Norm{w^{-1}\chi_B}_\cpp \leq K\mu(B). \]
The infimum over all such $K$ is called the $\App$ constant and is denoted $[w]_\App$.
\end{definition}

\begin{remark} If we adopt the usual convention that
\[ c \cdot \infty = \begin{cases} 0 & c = 0 \\ \infty & c > 0 \end{cases}, \]
if $w \in \App$, then $\Norm{w\chi_B}_\pp = \infty$ implies that $\Norm{w^{-1}\chi_B}_\cpp = 0$, and thus that $w^{-1}$ is the zero element in $L^{\cpp}$, contrary to $w$ being finite almost everywhere. Thus $w \in \Lpp$ and if $p_+ < \infty$ we can say that $w$ is locally $\pp$-integrable.
\end{remark}

In the case that $\pp$ is constant, $\pp=p \in (1,\infty)$, the $\App$ condition for $w$ is equivalent to Mucknhoupt's $\Ap$ condition for $w^p$. The necessity and sufficiency of the $\App$ condition for strong-type weighted norm inequalities of the maximal operator in $\mathbb{R}^n$ was first proved in \cite{pre} and simultaneously \cite{paper}. The following theorem, which is our main result, generalizes this to the case of spaces of homogeneous type.

\begin{theorem} \label{goal} Given $\pp \in \LH$ with $1 < p_- \leq p_+ < \infty$ and a weight $w$,
\[ \Norm{(Mf)w}_\pp \leq C\Norm{fw}_\pp \]
if and only if $w \in \App$.
\end{theorem}

By analogy with weak- and strong-type inequalities in classical $\Lp$ spaces, Theorem~\ref{goal} naturally suggests the following weak-type inequality, which remains an open problem.

\begin{conjecture} \label{conj}
Given $\pp \in \LH$ with $p_+ < \infty$ and a weight $w$,
\[ \Norm{t\chi_{\{x \in X\,:\,Mf(x)>t\}}w}_\pp \leq C\Norm{fw}_\pp \]
if and only if $w \in \App$.
\end{conjecture}

We will prove the necessity of the $\App$ condition for the weak-type inequality in Section~\ref{necessity}. Moreover, if $p_- > 1$ then the conjecture follows from Theorem~\ref{goal}. Conjecture~\ref{conj} is claimed to be true in $\mathbb{R}^n$ in \cite{paper}, but the proof contains a gap: if $p_- = 1$ then $(p')_+ = \infty$, and so Lemmas 3.3-6 (which are analogous to Lemmas~\ref{fracexp} and \ref{Ap to Ainfty} in this paper) may not be applied to $w^{-1} \in A_\cpp$, as is done several times throughout their proof. 

The remainder of this paper is devoted to proving Theorem~\ref{goal}. In Sections \ref{Lp}, \ref{Ap}, and \ref{cubes}, we will collect several elementary results about $\Lpp$ spaces, the $\App$ condition, and dyadic grids, respectively, on spaces of homogeneous type. In Section \ref{necessity} we will prove the necessity of the $\App$ condition, and in Section \ref{sufficiency} we will prove its sufficiency.

We adopt the convention throughout that $C$ denotes a large constant dependent only on fixed quantities (usually $X$, $\pp$, $w$, and the dyadic grid $\mathcal{D}$, unless otherwise stated or obvious from context). Multiples of balls are written as $CB(x,r) = B(x,Cr)$. By $A \approx B$, we mean that there are constants $c,C$ with $cB \leq A \leq CB$. Finally for a weight $w$ and a set $E$ we write $w(E) = \int_E w(x)\,d\mu$.


\section{Variable Lebesgue spaces} \label{Lp}
This section is a collection of elementary results regarding variable Lebesgue spaces on spaces of homogeneous type. We begin with two lemmas concerning spaces of homogeneous type which will be used to prove many of the results in this paper. The first is well-known, and we omit the proof. The second characterizes finite spaces of homogeneous type and is proved in \cite[Lemma 1.9]{facts}.

\begin{lemma}[Lower Mass Bound] \label{LMB}
There exists a positive constant $C=C(X)$ such that for all $x \in X$, $0<r<R$, and $y \in B(x,R)$,
\[ \frac{\mu(B(y,r))}{\mu(B(x,R))} \geq C\left(\frac{r}{R}\right)^{\log_2C_\mu}. \]
\end{lemma}

\begin{lemma}
A space of homogeneous type $X$ has $\mu(X) < \infty$ if and only if $\mu(X) = B(x,r)$ for some $x \in X$ and $r>0$.
\end{lemma}

The remainder of the lemmas in this section are facts which are well-known in $\mathbb{R}^n$. We omit proofs that are unchanged from their Euclidean case, which may be found in \cite{book, Diening}. We do, however, reproduce the proof of Lemma~\ref{normish}, as we later make reference to the constants implicit in the proof.

\begin{lemma} \label{normish} Given $\pp \in \mathcal{P}(X)$ with $p_+<\infty$, $\Norm{f}_\pp \leq C_1$ if and only if
\begin{equation} \int_X |f(x)|^{p(x)}d\mu \leq C_2. \label{modular} \end{equation}
Moreover, if one constant is equal to 1, we may take both to be.
\end{lemma}
\begin{proof}
Assume first that \eqref{modular} holds. Since $p_+<\infty$, we have that $\Norm{f}_{L^\infty(X_\infty)}=0$. Given $C_2 \leq 1$, we have $\rho(f/1) \leq 1$ and hence we may take $C_1 = 1$. If $C_2 \geq 1$, then we may divide to obtain
\[ \int_X \left|\frac{f(x)}{C_2^{1/p(x)}}\right|^{p(x)}d\mu \leq 1. \]
Now $C_2^{1/p(x)}$ is bounded by $C_1 \equiv C_2^{1/p_-}$, for which $\rho(f/C_1) \leq 1$ and so $\Norm{f}_\pp \leq C_1$.

Conversely, given $\Norm{f}_\pp \leq C_1$, then by the definition of the norm we get that
\[ 1 \geq \int_X\left|\frac{f(x)}{C_1+1}\right|^{p(x)}\,d\mu \geq \frac{1}{(C_1+1)^{p_+}}\int_X|f(x)|^{p(x)}\,d\mu, \]
and so that \eqref{modular} holds. If $C_1=1$, then for any $\epsilon>0$, there exists $\lambda_\epsilon \in [1,1+\epsilon)$ such that
\[ \rho(f/\lambda_\epsilon) = \int_X \left|\frac{f(x)}{\lambda_\epsilon}\right|^{p(x)}d\mu \leq 1. \]
Since the integrand is dominated by $|f(x)|^{p(x)}$, taking $\epsilon=1/n$ and applying the Dominated Convergence Theorem, we get that
\[ \int_X|f(x)|\,d\mu \leq 1, \]
and so $C_2$ may be taken to be 1.
\end{proof}

\begin{lemma} \label{unity}
Given $\pp \in \mathcal{P}(X)$ with $p_+<\infty$,
\[ \int_X \left(\frac{|f(x)|}{\Norm{f}_\pp}\right)^{p(x)}\,d\mu = 1. \]
In particular, if $\Norm{f}_\pp = 1$, then
\[ \int_X |f(x)|^{p(x)} d\mu = 1. \]
\end{lemma}

\begin{lemma} \label{conditionalbound}
Let $\pp \in \mathcal{P}(X)$ be such that $p_+<\infty$. If $\Norm{f}_\pp \leq 1$, then
\[ \Norm{f}_\pp^{p_+} \leq \int_X |f(x)|^{p(x)} d\mu \leq \Norm{f}_\pp^{p_-}. \]
On the other hand, if $\Norm{f}_\pp \geq 1$, then
\[ \Norm{f}_\pp^{p_-} \leq \int_X |f(x)|^{p(x)} d\mu \leq \Norm{f}_\pp^{p_+}. \]
\end{lemma}

\begin{lemma} \label{density}
If $\pp \in \mathcal{P}(X)$ is such that $p_+<\infty$, bounded functions with support contained in $B_r(x_0)$ for some $r$ and $x_0$ (bounded support) are dense in $\Lpp$. Moreover, any nonnegative $f \in \Lpp$ is the limit of an increasing sequence of such functions.
\end{lemma}
\begin{proof}
All bounded functions of bounded support are in $\Lpp$, because they are bounded by constant functions on finite-measure domains and $p_+<\infty$. To prove that such functions are dense, choose $f \in L^\pp$ and let $\epsilon>0$. By decomposing $f$ as
\[ f(x) = f_+(x)-f_-(x), \]
with both $f_+,f_- \geq 0$, it suffices to consider the case $f(x) \geq 0$. Since $\Norm{f}_\pp<\infty$, there exists $\Lambda>0$ for which $\rho(f/\Lambda) \leq 1$. If $\Lambda/\epsilon \equiv \lambda \geq 1$, then
\[ \int_X \left|\frac{f(x)}{\epsilon}\right|^{p(x)}d\mu = \int_X \left|\frac{\lambda f(x)}{\Lambda}\right|^{p(x)}d\mu \leq \lambda^{p_+}\int_X \left|\frac{f(x)}{\Lambda}\right|^{p(x)}d\mu \leq \lambda^{p_+} < \infty. \]
If $\lambda < 1$, the same argument holds with $p_-$ replacing $p_+$. Thus $\rho(f/\epsilon)<\infty$. Now define (choosing some base point $x_0 \in X$)
\[ f_n(x) = \min\{f(x),n\}\chi_{B(x_0,n)}. \]
It is clear that $f_n \to f$ pointwise. Then by the dominated convergence theorem,
\[ \int_X \left|\frac{f(x)-f_n(x)}{\epsilon}\right|^{p(x)}d\mu \to 0, \]
since $\rho(f/\epsilon)$ is finite and $\left|\frac{f(x)}{\epsilon}\right|^{p(x)}$ dominates the above integrand. But then for $n$ large enough that the above integral is less than one, $\Norm{f-f_n}_\pp \leq \epsilon$. It follows that the $f_n$ converge to $f$ in $\Lpp$ and consequently that bounded functions of bounded support are dense. Finally, we have that $f_{n+1}(x) \geq f_n(x)$, so the sequence increases to $f$.
\end{proof}

\begin{lemma}[Monotone Convergence Theorem] \label{MCT}
Given an exponent $\pp \in \mathcal{P}(X)$, let $\{f_k\}_{k=1}^\infty$ be a sequence of non-negative measureable functions that increase pointwise almost everywhere to a function $f \in \Lpp$. Then $\Norm{f_k}_\pp \to \Norm{f}_\pp$. 
\end{lemma}

\begin{lemma}[H\"{o}lder's Inequality] \label{Holder} Given an exponent $\pp \in \mathcal{P}(X)$,
\[ \int_X |f(x)g(x)|\,d\mu \leq 4\Norm{f}_\pp\Norm{g}_\cpp \]
for any $f,g$.
\end{lemma}

The following two lemmas are stated erroneously in \cite[Lemmas 2.7,\,2.8]{paper}. For clarity, we provide the correct proofs here.
\begin{lemma} \label{DCU2.8} Given a set $G \subseteq X$ and two exponents $s(\cdot)$ and $\rr$ such that
\[ 0 \leq r(y)-s(y) \leq \frac{C_0}{\log(e+d(x_0,y))} \]
for each $y \in G$, then for every $t \geq 1$ there exists a constant $C=C(t,C_0)$ such that for all functions $f$,
\begin{equation}\label{tms} \int_G |f(y)|^{s(y)}dy \leq C\int_G |f(y)|^{r(y)}dy + \int_G\frac{1}{(e+d(x_0,y))^{ts_-(G)}}dy. \end{equation}
\end{lemma}
\begin{proof}
Let $G' = \{y \in G\,:\,|f(y)| \geq (e+d(x_0,y))^{-t}\}$. Then decomposing the domain of the left integral in the inequality into $G'$ and $G \setminus G'$, we see that since $(e+d(x_0,y))^{-t} \leq 1$,
\[ \int_{G \setminus G'}|f(y)|^{s(y)}dy \leq \int_{G \setminus G'}(e+d(x_0,y))^{-ts(y)} dy \leq \int_G \frac{1}{(e+d(x_0,y))^{ts_-(G)}}dy. \]
If $y \in G'$, then
\begin{multline*} |f(y)|^{s(y)} = |f(y)|^{r(y)}|f(y)|^{s(y)-r(y)} \leq |f(y)|^{r(y)}(e+d(x_0,y))^{t(r(y)-s(y))} \\ \leq |f(y)|^{r(y)}(e+d(x_0,y))^{C_0t/\log(e+d(x_0,y))} \leq C|f(y)|^{r(y)}. \end{multline*}
The desired inequality follows.
\end{proof}

\begin{lemma} \label{DCU2.7} Given a set $G \subseteq X$ and two exponents $s(\cdot)$ and $r(\cdot)$ such that
\[ |r(y)-s(y)| \leq \frac{C_0}{\log(e+d(x_0,y))} \]
for each $y \in  G$, then for every $t \geq 1$ there exists a constant $C=C(t,C_0)$ such that \eqref{tms} holds for all functions $f$ with $|f(y)| \leq 1$.
\end{lemma}
\begin{proof}
Define the two sets $A=\{y \in G\,:\,r(y) \geq s(y)\}$ and $B=G \setminus A$. Lemma~\ref{DCU2.8} takes care of $A$. For $B$, construct $B' = \{y \in B\,:\,|f(y)| \geq (e+d(x_0,y))^{-t}\}$ and observe that the $B \setminus B'$ component holds as in the previous proof. But since $|f(y)| \leq 1$,
\[ |f(y)|^{s(y)} = |f(y)|^{r(y)}|f(y)|^{s(y)-r(y)} \leq |f(y)|^{r(y)}. \]
Since $C \geq 1$, this proves the inequality.
\end{proof}

Our proof of the final lemma in this section is based on the proof of \cite[Lemma 3.1]{Adamowicz}.

\begin{lemma} \label{DCU2.9} Given an exponent $\pp \in \text{LH}$, for all balls $B \subseteq X$,
\[ \mu(B)^{p_-(B)-p_+(B)} \leq C. \]
\end{lemma}
\begin{proof}
Fix $B=B(y_0,r)$ and define $B_0=B(x_0,1)$, where $x_0$ is the $\text{LH}_\infty$ condition base point. Also let $k = 2\ceil{\log_24A_0}+2$ and $C_1 = \mu(B_0)/C_\mu^k$. We will show that for any $x,y \in B$,
\[ \mu(B)^{-|p(x)-p(y)|} \leq C; \]
a simple limiting argument shows that this is equivalent to the stated form, by continuity of $\pp$.
If $\mu(B) \geq \min\{1,C_1\}$ then
\[ \mu(B)^{-|p(x)-p(y)|} \leq \min\{1,C_1\}^{-|p(x)-p(y)|} \leq \min\{1,C_1\}^{-p_+}, \]
since $\pp \in \LH$ implies $|p(x)-p(y_0)| \leq p_+ < \infty$. Thus we may assume $\mu(B) \leq \min\{1,C_1\}$.

We begin by asserting that we need only prove the inequality when one of the points is the center point $y_0$ of $B$. If this is not the case, then
\[ \mu(B)^{-|p(x)-p(y)|} = \mu(B)^{-|p(x)-p(y_0)+p(y_0)-p(y)|} \leq \mu(B)^{-|p(x)-p(y_0)|-|p(y)-p(y_0)|}, \]
and so it suffices to prove that
\[ \mu(B)^{-|p(x)-p(y_0)|} \leq C \]
for any $x \in B$.

We consider first the case where $r \geq 1$. For any $y \in r^{-1}B$, we have that
\[ d(x_0,y) \leq A_0(d(x_0,y_0) + d(y_0,y)) \leq A_0(1+d(x_0,y_0)). \]
Consequently $r^{-1}B \subseteq A_0(1+d(x_0,y_0))B_0$, and so by the lower mass bound (Lemma~\ref{LMB}),
\begin{multline*} \mu(B_0) \leq \mu(A_0(1+d(x_0,y_0))B_0) \\ \leq C(1+d(x_0,y_0))^{\log_2C_\mu}\mu(r^{-1}B) \leq C(1+d(x_0,y_0))^{\log_2C_\mu}\mu(B). \end{multline*}
Dividing by $\mu(B)$ and raising to the power of $|p(x)-p(y_0)|$, we get
\begin{equation} \mu(B)^{-|p(x)-p(y_0)|} \leq C(1+d(x_0,y_0))^{\log_2C_\mu|p(x)-p(y_0)|}. \label{Abound} \end{equation}
To estimate the quantity on the right of \eqref{Abound}, we argue that $B_0 \cap 2A_0B = \varnothing$. If to the contrary there exists $y \in B_0 \cap 2A_0B$, then for any $z \in B_0$ we have
\[ d(y_0,z) \leq A_0(d(y_0,y)+d(y,z)) \leq 2A_0(1+A_0r) \leq 4A_0^2r, \]
since $A_0,r \geq 1$. Consequently $B_0 \subseteq 4A_0^2B$. From the doubling condition,
\[ \mu(B) \geq \mu(4A_0^2B)/C_\mu^k \geq \mu(B_0)/C_\mu^k = C_1, \]
contrary to assumption. Hence the claim is true; in particular $x_0 \not\in 2A_0B$, so $d(x_0,y_0) > 2A_0r$. By the quasi-triangle inequality,
\[ d(x_0,y_0) \leq A_0(d(x_0,x) + d(x_0,x)) \leq A_0(r+d(y_0,x)). \]
Since $d(x_0,y_0) > 2A_0r$ we have that $d(x_0,x) > r$ and so $d(x_0,y_0) \leq 2A_0d(x_0,x)$. Thus
\[ (1+d(x_0,y_0))^{|p(x)-p(y_0)|} \leq (1+d(x_0,y_0))^{|p(y_0)-p_\infty|}(1+2A_0d(x_0,x))^{|p(x)-p_\infty|}. \]
That this is bounded by a constant is implied by the $\GLH$ condition and the fact that the function
\begin{equation} f(x) = \frac{\log(e+ax)}{\log(e+bx)} \label{function} \end{equation}
is bounded on $x>0$ by $\frac{a}{b}$ when $a>b>0$. This completes the case when $r \geq 1$.

If $r < 1$, we argue much as before with the lower mass bound to obtain
\[ \mu(B)^{-|p(x)-p(y_0)|} \leq Cr^{-|p(x)-p(y_0)|}\mu(r^{-1}B)^{-|p(x)-p(y_0)|}; \]
The $r^{-|p(x)-p(y_0)|}$ term is bounded by the $\LLH$ condition. 

\end{proof}


\section{The $\App$ Condition} \label{Ap}
In this section we develop the $\App$ condition in spaces of homogeneous type. Our first lemma characterizes various properties of $A_\infty$ weights. For a proof, see \cite[Chapter I, Theorem 15]{S&T}.
\begin{lemma} \label{Ainfty} Given a weight $W$, the following are equivalent.
\begin{itemize}
\item $W \in A_\infty = \bigcup_{p \geq 1} A_p$.
\item There exist constants $\epsilon > 0$ and $C_2 > 1$ such that given any ball $B$ and any measurable set $E \subseteq B$,
\[ \frac{\mu(E)}{\mu(B)} \leq C_2\left(\frac{W(E)}{W(B)}\right)^\epsilon. \]
\item $W$ is doubling (in the sense that the measure $\nu$ given by $d\nu(x) = W(x)\,d\mu(x)$ is doubling) and there exist constants $\delta > 0$ and $C_1 > 1$ such that given any ball $B$ and any measurable set $E \subseteq B$,
\[ \frac{W(E)}{W(B)} \leq C_1\left(\frac{\mu(E)}{\mu(B)}\right)^\delta. \]
\end{itemize}
\end{lemma}

To utilize the properties described in Lemma~\ref{Ainfty}, we will use the $\App$ condition to construct a weight $W$ in $A_\infty$. To do so, we require the following lemmas.

\begin{lemma} \label{normbound} Given an exponent $\pp \in \mathcal{P}(X)$, if $w \in \App$, then there exists a constant $C$ depending on $\pp$ and $w$ such that given any ball $B$ and any measurable set $E \subset B$,
\[ \frac{\mu(E)}{\mu(B)} \leq C\frac{\Norm{w\chi_E}_\pp}{\Norm{w\chi_B}_\pp}. \]
\end{lemma}
\begin{proof} Fix $B$ and $E \subset B$. By H\"{o}lder's inequality and the $\App$ condition (Definition~\ref{App}),
\begin{multline*} \mu(E) = \int_X w(x)\chi_Ew(x)^{-1}\chi_B\,d\mu \\ = C\Norm{w\chi_E}_\pp\Norm{w^{-1}\chi_B}_\cpp \leq C\Norm{w\chi_E}_\pp\Norm{w\chi_B}^{-1}_\pp\mu(B). \end{multline*}
\end{proof}

\begin{lemma} \label{fracexp}
Given an exponent $\pp \in LH$ and a weight $w \in \App$, there exists a constant $C_0$ depending on $\pp$, $w$, and $X$ such that for all balls $B$,
\[ \Norm{w\chi_B}_\pp^{p_-(B)-p_+(B)} \leq C_0. \]
\end{lemma}
\begin{proof}
Our proof is reminiscent of the proof of Lemma~\ref{DCU2.9}. Fix $B=B(y_0,r)$ and define $B_0=B(x_0,1)$. If $\Norm{w\chi_B}_\pp \geq 1$, then $\Norm{w\chi_B}_\pp^{p_-(B)-p_+(B)} \leq 1$, so we may assume that $\Norm{w\chi_B}_\pp<1$. We consider three cases; first, suppose $r \leq 1$ and $d(x_0,y_0) \leq 2A_0$. By the quasi-triangle inequality (Definition~\ref{quasi}), for any point $y \in B$, we have that
\[ d(x_0,y) \leq A_0(d(y_0,y) + d(x_0,y_0)) \leq A_0(r + 2A_0) \leq A_0(1+2A_0), \]
and so
\[ B \subseteq A_0(1+2A_0)B_0 \equiv B_1. \]
If we apply H\"{o}lder's inequality, the lower mass bound on $B_0$ and $B_1$, and the $\App$ condition, we get
\begin{multline} \mu(B) = \int_B w(x)w(x)^{-1}\,d\mu  \leq C\Norm{w\chi_B}_\pp\Norm{w^{-1}\chi_{B}}_\cpp  \\ \leq C\Norm{w\chi_B}_\pp\Norm{w^{-1}\chi_{B_1}}_\cpp (2A_0)^{\log_2C_\mu}\frac{\mu(B_0)}{\mu(B_1)}  \leq C\Norm{w\chi_B}_\pp\Norm{w\chi_{B_1}}^{-1}_\pp.\label{eq:1} \end{multline}
Here the constant depends on both $X$ and $x_0$. After rearranging, raising to the power $p_-(B)-p_+(B)$, and applying Lemma \ref{DCU2.9}, we obtain
\[ \Norm{w\chi_B}_\pp^{p_-(B)-p_+(B)} \leq C\mu(B)^{p_-(B)-p_+(B)}\Norm{w\chi_{B_q}}^{p_-(B)-p_+(B)} \leq  C(1+\Norm{w\chi_{B_1}}_\pp^{-1})^{p_+-p_-}, \]
which is a bound independent of $B$.

Consider now the case where $r > 1$ and $d(x_0,y_0) \leq 2A_0r$. Applying the quasi-triangle inequality as before,
\[ B_0 \subseteq A_0(1+2A_0r)B \equiv B_2. \]
Using H\"{o}lder's inequality and the $\App$ condition as in the previous case,
\begin{multline} \mu(B) 
\leq C\Norm{w\chi_B}_\pp\Norm{w^{-1}\chi_B}_{\cpp} \leq C\Norm{w\chi_B}_\pp\Norm{w^{-1}\chi_{B_2}}_\cpp \\ \leq C\Norm{w\chi_B}_\pp\mu(B_2)\Norm{w\chi_{B_2}}_\pp^{-1} \leq C\Norm{w\chi_B}_\pp\mu(B)\Norm{w\chi_{B_0}}_\pp^{-1}.\label{eq:2} \end{multline}
Thus
\[ \Norm{w\chi_B}_\pp^{p_-(B)-p_+(B)} \leq C(1+\Norm{w\chi_{B_0}}_\pp^{-1})^{p_+-p_-}. \]

Consider now the remaining case, namely when $d(x_0,y_0)>2A_0\max\{1,r\}$. Let $d=2A_0d(x_0,y_0)$ so that $B,B_0 \subseteq B(x_0,d) \equiv B_3$. Arguing as we did in inequality \eqref{eq:1} (if $1 \geq r$) or \eqref{eq:2} (if $r>1$) with $B_3$ in place of $B_1$ or $B_2$, we get
\[ \mu(B) \leq C\mu(B_3)\Norm{w\chi_B}_\pp\Norm{w\chi_{B_3}}_\pp^{-1}. \]
In order to bring $\mu(B_3)$ into the constant as in the previous cases and obtain the corresponding inequality, we need
\[ \mu(B_3)^{p_+(B)-p_-(B)} \leq C. \]
To see that this is the case, observe that $\pp$ is continuous (since it is in $\LLH$) and so there exist $y_1,y_2 \in \overline{B}$ for which $p(y_1)=p_-(B)$ and $p(y_2)=p_+(B)$. And since 
\[d(x_0,y_0) \leq A_0(d(x_0,y_k)+r) \leq A_0d(x_0,y_k)+\frac{1}{2}d(x_0,y_0)\] for $k=1,2$, we have that $d(x_0,y_k) \geq (2A_0)^{-1}d(x_0,y_0)$, so the $\GLH$ condition implies
\[ p_+(B)-p_-(B) \leq |p(y_1)-p_\infty| + |p(y_2)-p_\infty| \leq \frac{C}{\log(e+(2A_0)^{-1}d(x_0,y_0))}. \]
Using this together with the lower mass bound,
\[ \mu(B_3) \leq Cd^{\log_2C_\mu}\mu(B_0) = C(2A_0d(x_0,y_0))^{\log_2C\mu}\mu(B_0) \leq C(e+A_0d(x_0,y_0))^{log_2C_\mu}, \]
we get that
\begin{multline*} \mu(B_3)^{p_+(B)-p_-(B)} \leq \left[C(e+A_0d(x_0,y_0))^{log_2C\mu}\right]^{C/\log(e+(2A_0)^{-1}d(x_0,y_0))} \\ \leq Ce^{C\log_2C_\mu\log(e+A_0d(x_0,y_0))/\log(e+(2A_0)^{-1}d(x_0,y_0))} \leq Ce^{C\log_2C_\mu}. \end{multline*}
This last inequality is from the bound on \eqref{function}. Since the above bound is independent of $B$,
\[ [\mu(B)\Norm{w\chi_{B_3}}_\pp]^{p_+(B)-p_-(B)} \leq C \Norm{w\chi_B}_\pp^{p_+(B)-p_-(B)}. \]
If we apply Lemma \ref{DCU2.9} on the left and the bound just derived on the right, we obtain
\[ \Norm{w\chi_B}^{p_-(B)-p_+(B)}_\pp \leq C(1+\Norm{w\chi_{B_0}}^{-1})^{p_+-p_-}. \]

\end{proof}

We can now prove the following lemma, which will allow us to apply Lemma~\ref{Ainfty} to weights in variable exponent spaces.

\begin{lemma} \label{Ap to Ainfty}
Given an exponent $\pp \in LH$ and a weight $w \in \App$, we have that $W(\cdot) = w(\cdot)^\pp \in A_\infty$.
\end{lemma}
\begin{proof}
Fix a ball $B$ and a measurable set $E \subseteq B$. We will show that
\begin{equation}\label{pplus} \frac{\mu(E)}{\mu(B)} \leq C\left(\frac{W(E)}{W(B)}\right)^{1/p_+}, \end{equation}
which by Lemma \ref{Ainfty} is sufficient to show $W(\cdot) \in A_\infty$. We will prove this in three cases. Consider first the case that $\Norm{w\chi_B}_\pp \leq 1$. By Lemma \ref{normbound},
\begin{align*} \frac{\mu(E)}{\mu(B)} &\leq C\frac{\Norm{w\chi_E}_\pp}{\Norm{w\chi_B}_\pp} \leq C\frac{\Norm{w\chi_E}_\pp}{\Norm{w\chi_B}_\pp^{p_-(B)/p_+(B)}\Norm{w\chi_B}_\pp^{1-p_-(B)/p_+(B)}}. \end{align*}
If we appeal to Lemma \ref{conditionalbound} for the inequalities $\Norm{w\chi_E}_\pp \leq W(E)^{1/p_+(B)}$ and $\Norm{w\chi_B}_\pp^{p_-(B)} \geq W(E)$, then apply Lemma \ref{fracexp} on the remaining term, we get that
\[ \frac{\mu(E)}{\mu(B)} \leq C\left(\frac{W(E)}{W(B)}\right)^{1/p_+}\Norm{w\chi_B}^{p_-(B)/p_+(B)-1} \leq C\left(\frac{W(E)}{W(B)}\right)^{1/p_+}. \]
Now considering the case $\Norm{w\chi_E}_\pp \leq 1 \leq \Norm{w\chi_B}_\pp$, by the same lemmas as before,
\begin{multline*} \frac{\mu(E)}{\mu(B)} \leq C\frac{\Norm{w\chi_E}_\pp}{\Norm{w\chi_B}_\pp} \leq C\frac{\Norm{w\chi_E}_\pp}{\Norm{w\chi_B}_\pp^{p_-(B)/p_+(B)}\Norm{w\chi_B}_\pp^{1-p_-(B)/p_+(B)}} \\ \leq C\frac{W(E)^{1/p_+}}{\Norm{w\chi_B}^{p_-(B)/p_+(B)}\Norm{w\chi_B}^{1-p_-(B)/p_+(B)}},
\end{multline*}
which, given $\Norm{w\chi_B}_\pp \geq 1$ and $p_+ \geq p_+(B)$, yields
\[ \frac{\mu(E)}{\mu(B)} \leq C\left(\frac{W(E)}{W(B)}\right)^{1/p_+}. \]
The third case is $\Norm{w\chi_E}_\pp \geq 1$. Let $\lambda=\Norm{w\chi_B}_\pp \geq \Norm{w\chi_E}$. Since $\pp \in \GLH$, by Lemma \ref{DCU2.7} with $d\mu$ replaced by $W(x)\,d\mu$, for all $t>1$ there exists a constant $C_t$ for which
\begin{equation} \int_B \frac{W(x)}{\lambda^{p_\infty}}\,d\mu \leq C_t\int_B\frac{W(x)}{\lambda^{p(x)}}\,d\mu + \int_B\frac{W(x)}{(e+d(x_0,x))^{tp_\infty}}\,d\mu. \label{thirdcase} \end{equation}
The first integral on the right hand side is less than 1 by Lemma \ref{unity}. We claim that the same is true of the second term for sufficiently large $t$ independent of $B$. This is obvious if $W(X) < \infty$, since
\[ \int_X\frac{W(x)}{(e+d(x_0,x))^{tp_\infty}}\,d\mu \leq Ce^{-tp_\infty}W(X), \]
which may be made arbitrarily small. If on the other hand $W(X)=\infty$, let $B_k=B(x_0,2^k)$. Then by Lemma \ref{conditionalbound},
\begin{align*} \int_X\frac{W(x)}{(e+d(x_0,x))^{tp_\infty}}d\mu &\leq e^{-tp_\infty}W(B_0)+C\sum_{k=1}^\infty \int_{B_k \setminus B_{k-1}}\frac{W(x)}{(e+d(x_0,x))^{tp_\infty}}\,d\mu \\ &\leq e^{-tp_\infty}W(B_0)+C\sum_{k=1}^\infty 2^{-ktp_\infty}W(B_k) \\ &\leq e^{-tp_\infty}W(B_0)+C\sum_{k=1}^\infty 2^{-ktp_\infty}\max\bigg\{\Norm{w\chi_{B_k}}_\pp^{p_+},\Norm{w\chi_{B_k}}_\pp^{p_-}\bigg\} \\
&\leq e^{-tp_\infty}W(B_0) + C\sum_{k=1}^\infty 2^{-ktp_\infty}\Norm{w\chi_{B_k}}_\pp^{p_+}.
\end{align*}
The last inequality comes from the fact that $\Norm{w\chi_{B_k}}_\pp>1$ for all $k$ sufficiently large, by continuity of the measure $dW=W(x)\,d\mu$ and the fact that $X = \bigcup_{k=1}^\infty B_k$. By Lemma \ref{normbound},
\[ \Norm{w\chi_{B_k}}_\pp \leq C\frac{\mu(B_k)}{\mu(B_0)}\Norm{w\chi_{B_0}}_\pp \leq C2^{k\log_2C_\mu}. \]
Combining these two estimates yields
\begin{equation} \label{tsum} \int_X \frac{W(x)}{(e+d(x_0,x))^{tp_\infty}}d\mu \leq e^{-tp_\infty}W(B_0)+C\sum_{k=1}^\infty 2^{kp_+\log_2C_\mu-ktp_\infty}. \end{equation}
For $t>p_\infty/\log_2C_\mu^{p_+}$ the sum converges, and choosing $t$ sufficiently large (independent of $B$) makes the right hand side less than 1. Thus the right hand side of \eqref{thirdcase} is bounded, and so we may rearrange to obtain
\begin{equation}\label{sumboundop} W(B)^{1/p_\infty} \leq C\Norm{w\chi_B}_\pp. \end{equation}

Now repeating the argument switching $B$ with $E$ and $\pp$ with $p_\infty$, we get
\[ 1 \leq \int_E \frac{W(x)}{\lambda^{p(x)}}\,d\mu \leq C_t\int_E\frac{W(x)}{\lambda^{-p_\infty}}\,d\mu + \int_E\frac{W(x)}{(e+d(x_0,x))^{tp_\infty}}\,d\mu. \]
As before, we can make the rightmost term less than $1/2$, so that
\begin{equation} \lambda^{p_\infty} = \Norm{w\chi_E}_\pp^{p_\infty} \leq CW(E).  \label{sumbound} \end{equation}
Then by Lemma \ref{normbound},
\[ \frac{\mu(E)}{\mu(B)} \leq C\frac{\Norm{w\chi_E}_\pp}{\Norm{w\chi_Bs}_\pp} \leq C\left(\frac{W(E)}{W(B)}\right)^{1/{p_\infty}} \leq C\left(\frac{W(E)}{W(B)}\right)^{1/{p_+}}. \]
\end{proof}

From the latter stages of the proof of Lemma~\ref{Ap to Ainfty}, we may pull the following corollary.
\begin{corollary} \label{cor}
Given an exponent $\pp \in \LH$, if $w \in \App$ is a weight satisfying $\Norm{w\chi_B}_\pp \geq 1$ on a ball $B$, then $\Norm{w\chi_B}_\pp \approx W(B)^{1/p_\infty}$.
\end{corollary}

We conclude this section with a lemma that will allow for the reduction from our main result to the unweighted case.

\begin{lemma} If $\pp \in \LH$ and $p_- > 1$, then $1 \in \App$. \end{lemma}
\begin{proof} Fix a ball $B$. If $\mu(B) \leq 1$, then by Lemma \ref{conditionalbound},
\[ \Norm{\chi_B}^{p_+(B)}_\pp \leq \int_B 1^{p(x)}d\mu = \mu(B), \]
which implies
\[ \Norm{\chi_B}_\pp \leq C\mu(B)^{1/p_+(B)}, \]
and by the same argument applied to $\cpp$,
\[ \Norm{\chi_B}_\cpp \leq C\mu(B)^{1/(p')_+(B)} = C\mu(B)^{1-1/p_-(B)}. \]
Thus (applying Lemma \ref{DCU2.9})
\[ \Norm{\chi_B}_\pp\Norm{\chi_B}_\cpp \leq C[\mu(B)^{p_-(B)+p_+(B)}]^{1/p_+(B)p_-(B)}\mu(B) \leq K\mu(B), \]
which is the desired inequality.
Suppose now $\mu(B) > 1$. By an argument that is essentially the same as the proof of Corollary~\ref{cor} with $w=1$, we get that
\[ \Norm{\chi_B}_\pp \Norm{\chi_B}_\cpp \leq K\mu(B)^{1/{p_\infty}+1/{p'_\infty}} = K\mu(B). \]
\end{proof}


\section{Dyadic Cubes} \label{cubes}
Important to the proofs of many results of variable exponent spaces in $\mathbb{R}^n$ are the dyadic cubes of the form
\[ Q = [m_12^{-k},(m_1+1)2^{-k}) \times \cdots \times [m_n2^{-k},(m_n+1)2^{-k}), \quad m_1,\dots,m_n \in \mathbb{Z}. \]
Due to the usefulness of dyadic objects in many areas of harmonic analysis, a great deal of effort has gone into developing similar systems in metric and quasi-metric spaces, for example \cite{Christ,systems,Korvenpaa}. We will use the form of Hy\"oten and Kairema's construction \cite{systems} presented in \cite{bump}.

\begin{theorem} \label{dyadic} There exist constants $C_d>0,d_0>1$, and $0<\epsilon<1$ depending on $X$, a family $\mathcal{D} = \bigcup_{k \in \mathbb{Z}} \mathcal{D}_k$, called the \textbf{dyadic grid} on $X$ of subsets of $X$, called \textbf{dyadic cubes}, and a collection $\{x_c(Q)\}_{Q \in \mathcal{D}}$ of points such that:
\begin{enumerate}
\item For every $k \in \mathbb{Z}$ the cubes in $\D_k$ are pairwise disjoint and $X = \bigcup_{Q \in \mathcal{D}_k} Q$. We will refer to the cubes in $\mathcal{D}_k$ as cubes in the $k$\emph{th generation}.
\item If $Q_1,Q_2 \in \mathcal{D}$, then either $Q_1 \cap Q_2 = \varnothing$, or $Q_1 \subseteq Q_2$, or $Q_2 \subseteq Q_1$.
\item For any $Q_1 \in \mathcal{D}_k$, there exists at least one $Q_2 \in  \mathcal{D}_{k-1}$, which is called a \emph{child} of $Q_1$, such that $Q_2 \subseteq Q_1$, and there exists exactly one $Q_3 \in \mathcal{D}_{k+1}$, which is called a $\emph{parent}$ of $Q_1$, such that $Q_1 \subseteq Q_3$.
\item If $Q_2$ is a child of $Q_1$, then $\mu(Q_2) \geq \epsilon\mu(Q_1)$.
\item For every $k$ and $Q \in \mathcal{D}_k$, $B(x_c(Q),d_0^k) \subseteq Q \subseteq B(x_c(Q),C_dd_0^k)$.
\end{enumerate}
\end{theorem}

In general, we may freely switch back and forth between the settings of cubes and balls. Consider, for example, the following equivalent formulation of the $\App$ condition.

\begin{lemma}[The $\App$ condition for cubes] \label{Apcubes}
Given a dyadic grid $\mathcal{D}$ and $\pp \in \LH$, if $w \in \App$, then there exists a constant $K_q$ such that for any $Q \in \mathcal{D}$,
\[ \Norm{w\chi_Q}_\pp \Norm{w^{-1}\chi_Q}_\cpp \leq K\mu(Q). \]
\end{lemma}
\begin{proof} Fix $Q \in \mathcal{D}_k$. Then by Theorem \ref{dyadic}, the $\App$ condition, and the lower mass bound,
\begin{multline*}
\Norm{w\chi_Q}_\pp\Norm{w^{-1}\chi_Q}_\cpp \leq \Norm{w\chi_{B(x_c(Q),C_dd_0^k)}}_\pp \Norm{w^{-1}\chi_{B(x_c(Q),C_drd_0^k)}}_\cpp \\
\leq K\mu(B(x_c(Q),Cd_0^k)) \leq C\mu(B(x_c(Q),d_0^k)) \leq C\mu(Q).
\end{multline*}
The constant $C$ is independent of $k$.
\end{proof}

In general, the argument in the proof of Lemma \ref{Apcubes}, in which we expand cubes to fill balls and then apply the lower mass bound to shrink back to cubes, may be used to show that any previously stated result is also true when balls are replaced by cubes. In particular, Lemmas~\ref{Ainfty} and \ref{fracexp} hold in this way. Another object which it is convenient to recast into a dyadic form is the maximal operator.

\begin{definition}
Given a weight $\sigma$ and a dyadic grid $\mathcal{D}$, define the \textbf{weighted dyadic maximal operator} $M_\sigma^\mathcal{D}$ with respect to $\mathcal{D}$ by
\[ M_\sigma^\mathcal{D} f(x) = \sup_{\substack{Q \ni x \\ Q \in \mathcal{D}}} \, \avgint_Q |f(y)|\,d\sigma \]
for any locally integrable function $f$. When $\sigma = 1$, we will denote $M^\mathcal{D}_\sigma$ simply by $M^\mathcal{D}$.
\end{definition}

The weighted dyadic maximal operator satisfies the same weak- and strong-type inequalities as the classical maximal operator. Given a fixed grid $\mathcal{D}$ and weight $\sigma$, for each $\lambda>0$, we define the set
\[ X^\mathcal{D}_\lambda = \{x \in X\,:\,M^\mathcal{D}_\sigma f(x) > \lambda\}. \]
Then the following lemma holds.

\begin{lemma} \label{pbound}
Given a dyadic grid $\mathcal{D}$ on $X$ and a weight $\sigma$, the dyadic maximal operator $M_\sigma^\mathcal{D}$ is weak $(1,1)$: for $f \in L^1(\sigma)$ and all $\lambda>0$,
\[ \sigma\left(X^\mathcal{D}_\lambda\right) \leq \frac{1}{\lambda}\int_X |f(x)|\,d\sigma. \]
Further, for $1<p<\infty$, $M_\sigma^\mathcal{D}$ is strong (p,p): there exists a constant $C$ depending on $p$ and $X$ such that for any $f \in L^p(\sigma)$,
\[ \int_X M_\sigma^\mathcal{D} f(x)^p\,d\sigma \leq C\int_X |f(x)|^p\,d\sigma. \]
\end{lemma}
\begin{proof}
For each integer $n$, define the truncated maximal operator
\[ M_\sigma^n f(x) = \sup_{\substack{x \in Q \in \mathcal{D}_k \\ k \leq n}} \, \avgint_Q |f(y)|\,d\sigma. \]
Observe that for every $x \in X$, the sequence $\left\{M_\sigma^k f(x)\right\}$ increases to $M_\sigma^\mathcal{D} f(x)$. Certainly, it is increasing and bounded; if $M_\sigma^\mathcal{D} f(x)<\infty$, then for any $\epsilon>0$ there exists a cube $Q$ for which
\[ M_\sigma^\mathcal{D} f(x)-\epsilon \leq \avgint_Q |f(y)|\,d\sigma; \]
but then for any $n$ greater than the generation of $Q$,
\[ \avgint_Q |f(y)|\,d\sigma \leq M_\sigma^n f(x), \]
and so the sequence converges. A similar argument shows that if $M_\sigma^\mathcal{D} f(x) = \infty$, then $M_\sigma^n f(x)$ can be made greater than any integer.

Therefore, by the monotone convergence theorem, it suffices to prove the weak-type inequality for the truncated maximal operator. To that end, fix $\lambda>0$. If $M_\sigma^n f(x)>\lambda$, then there exists a cube $Q_x$ containing $x$ such that
\[ \avgint_{Q_x} |f(y)|\,d\sigma > \lambda, \]
and $Q_x$ is of generation at most $n$. Without loss of generality, take $Q_x$ to be the maximal of all such cubes, and let its generation be $k$. Since there are countably many dyadic cubes, the set $\{Q_x\,:\,x \in X\}$ may be enumerated as $\{Q_j\}$. If $Q_i \cap Q_j \neq \varnothing$ for some $i \neq j$, then we have some containment $Q_i \subseteq Q_j$ (without loss of generality), and thus $Q_i=Q_j$ by maximality, so the cubes are mutually disjoint. Then
\[ \sigma\left(\{x \in X\,:\,M_\sigma^n f(x)>\lambda\}\right) = \sum_j \sigma(Q_j) \leq \frac{1}{\lambda}\sum_j \int_{Q_j} |f(y)|\,d\sigma \leq \int_X |f(y)|\,d\sigma. \]
This proves the weak-type inequality.

For the strong-type inequality,
\[ \avgint_Q |f(y)|\,d\sigma \leq \frac{1}{\sigma(Q)}\Norm{f}_{L^\infty(\sigma)} \int_Q \,d\sigma = \Norm{f}_{L^\infty(\sigma)} = \Norm{f}_{L^\infty(\sigma)}. \]
Now fix $1<p<\infty$ and $ f \in L^1(\sigma) \cap L^\infty(\sigma)$. Without loss of generality, assume $f\sigma \neq 0$. Then $M^\mathcal{D}_\sigma f \in L^{1,\infty}(\sigma) \cap L^\infty(\sigma)$,
and consequently by Tonelli's theorem,
\[ \int_X M^\mathcal{D}f(x)^p\,d\sigma = \int_0^\infty p\lambda^{p-1}\sigma\left(\{x \in X\,:\,M^\mathcal{D}_\sigma f(x) > \lambda\}\right)\,d\lambda \leq C\int_0^{\Norm{M^\mathcal{D}_\sigma f}_{L^\infty(\sigma)}}\lambda^{p-2}\,d\lambda < \infty. \]
Thus $0 < \Norm{M^\mathcal{D}_\sigma f}_{L^p(\sigma)}<\infty$. Hence, by the weak-type inequality, Tonelli's Theorem, and H\"{o}lder's inequality,
\begin{align*} \int_X M_\sigma^\mathcal{D} f(x)^p\sigma(x)\,d\mu &= p\int_0^\infty \lambda^{p-1}\sigma\left(\{x \in X\,:\,M_\sigma^\mathcal{D} f(x) > \lambda\}\right)\,d\lambda \\ &\leq p\int_0^\infty \lambda^{p-2}\int_X|f(x)|\,d\sigma\,d\lambda \\ &= p\int_X |f(x)|\int_{\{\lambda\,:\,M^\mathcal{D}_\sigma f(x)>\lambda\}} \lambda^{p-2}\,d\lambda\,\,d\sigma \\
&\leq \frac{p}{p-1}\int_X |f(x)|[M_\sigma^\mathcal{D} f(x)]^{p-1}\,d\sigma \\ &\leq C\Norm{f}_{L^p(\sigma)}\Norm{M_\sigma^\mathcal{D}f}_{L^p(\sigma)}^{p-1}. \end{align*}
Rearranging, we obtain that
\[ \int_X M_\sigma^\mathcal{D} f(x)^p\,d\sigma \leq C\int_X|f(x)|^p\,d\sigma, \]
which is the desired strong-type inequality. For general functions $f \in L^p(X)$, the desired inequality follows from an approximation argument if we use Lemma \ref{density} and the monotone convergence theorem.
\end{proof}

We now prove the Calder\'on-Zygmund decomposition for the maximal operator over spaces of homogeneous type. This result is known, but since we could not find the precise formulation we wanted, for completeness we include the proof here.

\begin{lemma}[Calder\'{o}n-Zygmund Decomposition]\label{CZ} If $\mu(X)=\infty$, given a weight $\sigma \in A_\infty$, let $\mathcal{D}$ be a dyadic grid on $X$. If $f \in L^1_{loc}(\sigma)$ is such that $\avgint_{Q^k} |f(x)|\sigma(x)\,d\mu \to 0$ for any nested sequence $\{Q_k \in \mathcal{D}_{k}\}_{k=0}^\infty$, where each $Q_k$ is a child of $Q_{k+1}$, then for each $\lambda>0$, there exists a (possibly empty) set $\{Q_j\}$, called the Calder\'on-Zygmund (CZ) cubes of $f$ at height $\lambda$, of pairwise disjoint dyadic cubes and a constant $C_{CZ}=C_{CZ}(\mathcal{D},X,\sigma)>1$, independent of $\lambda$, such that
\[ X^\mathcal{D}_\lambda = \bigcup_j Q_j. \]
Moreover, for each $j$,
\begin{equation} \lambda < \avgint_{Q_j} |f(x)|\,d\sigma \leq C_{CZ}\lambda. \label{CZdecomp} \end{equation}
If $\{Q^k_j\}$ are the Calder\'on-Zygmund cubes at height $a^k$ for $k \in \mathbb{Z}$ and $a>C_{CZ}$, define $E^k_j=Q^k_j \setminus X^\mathcal{D}_{a^{k+1}}$. These sets are pairwise disjoint for all $j$ and $k$, and $\sigma(E^k_j) \geq \frac{a-C_{CZ}}{a}\sigma(Q^k_j)$.

If $\mu(X) < \infty$, then the Calder\'on-Zygmund cubes may be constructed for any function $f \in L^1_{loc}(\sigma)$ and at any height $\lambda > \avgint_X |f(y)|\,d\sigma \equiv \lambda_0$, with \eqref{CZdecomp} still holding. In this case, the sets $E^k_j$ are defined only for $k > \log_a\lambda_0$, and are pairwise disjoint with $\sigma(E^k_j) \geq \frac{a-C_{CZ}}{a}\sigma(Q^k_j)$.
\end{lemma}

\begin{proof}
Suppose first $\mu(X)=\infty$ and fix $\lambda > 0$. If $X^\mathcal{D}_\lambda$ is empty, then take $\{Q_j\}$ to be the empty set. Otherwise, fix $x \in X^\mathcal{D}_\lambda$. Then $x$ is contained in exactly one cube $Q^x_k$ of each generation $k$ and $M^\mathcal{D}_\sigma f(x) > \lambda$, so there exists at least one $k$ for which
\begin{equation} \avgint_{Q^x_k} |f(y)|\,d\sigma > \lambda. \label{max} \end{equation}
Since by assumption
\[ \lim_{k \to \infty} \avgint_{Q^x_k} |f(y)|\,d\sigma \to 0, \]
we may take $k$ to be the largest integer for which \eqref{max} holds. Let $\{Q_x\,:\,x \in X^\mathcal{D}_\lambda\}$ be the set of all such maximal cubes. As in the proof of Lemma~\ref{pbound}, this set must be countable and mutually disjoint. Clearly, $X^\mathcal{D}_\lambda$ is contained in the union of these cubes. Conversely, given any $z \in Q_x$ for some $x$, we have that
\[ M^\mathcal{D}_\sigma f(z) \geq \avgint_{Q_x} |f(y)|\,d\sigma > \lambda, \]
and so $z \in X^\mathcal{D}_\lambda$; consequently,
\[ X^\mathcal{D}_\lambda = \bigcup_j Q_j. \]
We now wish to show the inequalities \eqref{CZdecomp}. The first holds by choice of $Q_j$. For the second, the maximality of each $Q_j$ ensures that its parent, $\widehat{Q}_j$, satisfies
\[ \avgint_{\widehat{Q}_j} |f(y)|\,d\sigma \leq \lambda. \]
From this fact together with Lemma \ref{dyadic} and the lower mass bound,
\[ \avgint_{Q_j} |f(y)|\sigma(y)\,d\mu \leq \frac{\sigma(\widehat{Q}_j)}{\sigma(Q_j)}\lambda \leq \frac{\sigma(B(x_c(\widehat{Q}_j),Cd_0^{k+1}))}{\sigma(B(x_c(Q_j),d_0^k))}\lambda \leq Cd_0^{\log_2C_\mu}\lambda, \]
which is the second inequality in \eqref{CZdecomp}.

Now fix $a>C_{CZ}$ and consider the Calder\'on-Zygmund cubes $\{Q_j^k\}$ at heights $a^k$ for $k \in \mathbb{Z}$. For simplicity, we define $X_k \equiv X^\mathcal{D}_{a^k}$. Observe that $X_{k+1} \subset X_k$. Consequently, given any $Q_i^{k+1}$, the set $\{Q^x_k\}$ (constructed above) for an arbitrary $x \in Q_i^{k+1}$ contains $Q_i^{k+1}$, and so there exists $j$ such that $Q_i^{k+1} \subset Q_j^k$. 

We claim that this implies that the sets $E^k_j$ are pairwise disjoint for all $j,k$. To see this, consider two arbitrary sets $E^{k_1}_{j_1}$ and $E^{k_2}_{j_2}$ and suppose without loss of generality that $k_1 \leq k_2$. By the above argument, there exists $j_3$ such that $Q^{k_2}_{j_2} \subset Q^{k_1}_{j_3}$. If $j_3 = j_1$, then $k_1 \neq k_2$ and so disjointness arises from the containment $E^{k_2}_{j_2} \subset X_{k_2} \subset X_{k_1}$; otherwise, the disjointness of $Q^k_j$ for fixed $k$ implies that for $E^{k_1}_{j_1}$ and $E^{k_2}_{j_2}$.

Now fix $Q^k_j$; we have that
\begin{equation} \sigma(Q^k_j) = \sigma(Q^k_j \cap X_{k+1})+\sigma(E^k_j). \label{sparsedecomp}\end{equation}
By the properties listed above,
\begin{multline*} \sigma(Q^k_j \cap X_{k+1}) = \sum_{i:Q^{k+1}_i \subset Q^k_j} \sigma(Q^{k+1}_i) \\ \leq \frac{1}{a^{k+1}}\sum_{i:Q^{k+1}_i \subset Q^k_j}\int_{Q^{k+1}_i}|f(y)|\,d\sigma
\leq \frac{1}{a^{k+1}}\int_{Q^k_j}|f(y)|\,d\sigma \leq \frac{C_{CZ}}{a}\sigma(Q^k_j). \end{multline*}
After plugging this into \eqref{sparsedecomp} and rearranging, we obtain
\[ \sigma(E^k_j) \geq \frac{a-C_{CZ}}{a}\sigma(Q^k_j), \]
which is the desired inequality.

For $\mu(X)<\infty$, the proof is the same, with one exception. Since $X$ is bounded, for all cubes $Q$ sufficiently large, $Q=X$. As such, choosing $\lambda > \avgint_X |f(y)|\,d\sigma$ ensures that we may find maximal cubes as before.
\end{proof}


\section{Necessity} \label{necessity}
In this section we prove the necessity of the $\App$ condition in Theorem~\ref{goal}. Actually, we will prove necessity in Conjecture~\ref{conj}, but by monotonicity of the norm, we get that the strong-type inequality implies the weak-type, so to prove necessity in both results it suffices to demonstrate that any weight satisfying the latter is in $\App$.

To that end, let $w$ be such a weight and fix a ball $B \subseteq X$.
First, we will show that $w$ is $\pp$-integrable on $B$. Supposing to the contrary, since $p_+ < \infty$ we have from Lemma~\ref{normish} that $\Norm{w\chi_B}_\pp = \infty$. Fix $x \in B$ and choose any ball $E$ with $x \in E \subseteq B$. If we choose $f = \chi_E$ then $Mf(x) \geq \frac{\mu(E)}{\mu(B)}\chi_B$. Then for each $t < \frac{\mu(E)}{\mu(B)}$ the weak-type inequality implies that
\[ t\Norm{w\chi_B}_\pp \leq \Norm{t\chi_{\{x\in X\,:\,Mf(x)>t\}}w}_\pp \leq C\Norm{w\chi_E}_\pp. \]
Thus the right hand side must be infinite, and so by Lemma~\ref{normish},
\[ \int_E w(x)^{p(x)}\,d\mu = \infty. \]
Letting $E$ shrink to $x$ and applying the Lebesgue Differentiation Theorem (since $\mu$ is Borel regular; see \cite[Theorem 1.4]{Ahlfors}), we find that $w(x)^{p(x)} = \infty$ and thus $w(x)=\infty$ for almost every $x$, contrary to the definition of a weight. It follows that $w$ is locally $\pp$-integrable.

Now we show that $w \in \App$. We first assume that $\Norm{w^{-1}\chi_B}_\cpp < \infty$; later, we will see that this is necessarily the case. By the homogeneity of both the weak-type inequality and the $\App$ condition in $w$, we can assume that $\Norm{w^{-1}\chi_B}_\cpp = 1$.

We partition $B$ into the sets
\[ F_0 \equiv \{x \in B\,:\, p'(x) < \infty\}, \qquad F_\infty \equiv \{x \in B\,:\,p'(x) = \infty\}. \]
By the definition of the norm, for any $\lambda \in (\frac{1}{2},1)$,
\[ 1 < \rho_{\cpp}\left(\frac{w^{-1}\chi_B}{\lambda}\right) = \int_{F_0}\left(\frac{w(x)^{-1}}{\lambda}\right)^{p'(x)}\,d\mu + \lambda^{-1}\Norm{w^{-1}\chi_{F_\infty}}_\infty. \]
One of the terms on the right must be greater than $\frac{1}{2}$. More specifically, one of the following must be true: either $\Norm{w^{-1}\chi_{F_\infty}}_\infty \geq \frac{1}{2}$, or there exists $\lambda_0 \in (\frac{1}{2},1)$ for which $\int_{F_0}\left(\frac{w(x)^{-1}}{\lambda}\right)^{p'(x)}\,d\mu \geq \frac{1}{2}$ for any $\lambda \in [\lambda_0,1)$.
Suppose for now it is the first.

Fix $s > \Norm{w^{-1}\chi_{F_\infty}}_\infty^{-1} = \essinf_{x \in F_\infty}w(x)$. There exists a subset $E \subseteq F_\infty$ with $\mu(E)>0$ such that $w(E) \subseteq (0,s]$. Choose the function $f = \chi_E$. Since $\pp$ is identically 1 on $F_\infty$,
\[ \Norm{fw}_\pp = \Norm{w\chi_E}_\pp = w(E). \]
Further, we see that for all $x \in B$,
\[ Mf(x) \geq \frac{\mu(E)}{\mu(B)}. \]
Thus if we fix $t < \frac{\mu(E)}{\mu(B)}$, the weak-type inequality implies that
\[ t\Norm{w\chi_B}_\pp \leq t\Norm{w\chi_{\{x\,:\,Mf(x)>t\}}}_\pp \leq C\Norm{fw}_\pp = Cw(E). \]
If we take the supremum over all such $t$ and rearrange, we get that
\[ \frac{1}{\mu(B)}\Norm{w\chi_B}_\pp \leq C\frac{w(E)}{\mu(E)} \leq Cs. \]
Now taking the infimum over all such $s$, we get
\[ \frac{1}{\mu(B)}\Norm{w\chi_B}_\pp \leq C\Norm{w^{-1}\chi_{F_\infty}}_\infty^{-1} \leq 2C. \]
Since $\Norm{w^{-1}\chi_B}_\cpp = 1$, this is the $\App$ condition on $B$.

We now consider the case that
\[ \int_{F_0}\left(\frac{w(x)^{-1}}{\lambda}\right)^{p'(x)}\,d\mu \geq \frac{1}{2} \]
for all $\lambda \in [\lambda_0,1)$. If we define $F_R = \{x \in F_0\,:\,p'(x) < R\}$ for $R>1$, by the monotone convergence theorem for $\Lpp$ norms (Lemma~\ref{MCT}) we may find $R$ sufficiently large that
\[ \int_{F_R} \left(\frac{w(x)^{-1}}{\lambda_0}\right)^{p'(x)}\,d\mu > \frac{1}{3}.\]
Further, since $\Norm{w^{-1}\chi_B}_\cpp=1$, by Lemma~\ref{normish},
\begin{align*}
\int_{F_R}\left(\frac{w(x)^{-1}}{\lambda_0}\right)^{p'(x)}\,d\mu &\leq \int_{F_R}\left(\frac{2}{\lambda_0}\right)^{p'(x)}\left(\frac{w(x)^{-1}}{2}\right)^{p'(x)}\,d\mu \\ &\leq \left(\frac{2}{\lambda_0}\right)^R \int_{F_R}\left(\frac{w(x)^{-1}}{2}\right)^{p'(x)}\,d\mu \\ &\leq \left(\frac{2}{\lambda_0}\right)^R \\ &< \infty.
\end{align*}
Now define the function
\[ G(\lambda) = \int_{F_R} \left(\frac{w(x)^{-1}}{\lambda}\right)^{p'(x)}\,d\mu. \]
Then we know from the above computations that $\frac{1}{3} < G(\lambda_0) < \infty$ and by the dominated convergence theorem that $G$ is continuous on $[\lambda_0,1]$. If $G(1) \geq \frac{1}{3}$, then by Lemma~\ref{normish}, for any $\lambda \in [\lambda_0,1)$,
\[ \frac{1}{3\lambda} \leq \frac{1}{\lambda}\int_{F_R}w(x)^{-p'(x)}\,d\mu \leq G(\lambda) \leq \lambda^{-R} < \infty. \]
Now by taking $\lambda$ sufficiently close to 1, we may make $\lambda^{-R} \leq 2$, so that
\begin{equation} \frac{1}{3} \leq \int_{F_R}\left(\frac{w(x)^{-1}}{\lambda}\right)^{p'(x)}\,d\mu \leq 2. \label{finally} \end{equation}
On the other hand, if $G(1) < \frac{1}{3}$, then by continuity there is some $\lambda \in (\lambda_0,1)$ such that $G(\lambda)=\frac{1}{3}$, and so by choosing this $\lambda$ we get that \eqref{finally} holds in this case as well.

Having fixed $\lambda$, we now choose our function to be
\[ f(x) = \frac{w(x)^{-p'(x)}}{\lambda^{p'(x)-1}}\chi_{F_R}. \]
Then
\[ \rho_{\pp}(fw) = \int_{F_R}\left(\frac{w(x)^{-1}}{\lambda}\right)^{p'(x)}\,d\mu  \leq 2. \]
Hence, by the proof of Lemma~\ref{normish}, $\Norm{fw}_\pp \leq 2^{1/(p')_-}$. On the other hand, for all $x \in B$,
\[ Mf(x) \geq \avgint_B f(x)\,d\mu =  \frac{\lambda}{\mu(B)}\int_{F_R}\left(\frac{w(x)^{-1}}{\lambda}\right)^{p'(x)}\,d\mu \geq \frac{\lambda}{3\mu(B)}. \]
Thus for $t < \frac{\lambda}{3\mu(B)}$, by the weak-type inequality,
\[ C \geq C\Norm{fw}_\pp \geq t\Norm{w\chi_{\{x\,:\,Mf(x)>t\}}}_\pp \geq t\Norm{w\chi_B}_\pp, \]
which after taking the supremum over all such $t$ is the $\App$ condition.

It remains to be shown that $w \in \App$ if $\Norm{w^{-1}\chi_B}_\cpp = \infty$. To that end, fix $\epsilon > 0$ and define the weight $w_\epsilon(x) = w(x)+\epsilon$. Then $w_\epsilon^{-1} \leq \epsilon^{-1} < \infty$ and so $\Norm{w_\epsilon^{-1}\chi_B}_\cpp < \infty$. We observe that
\begin{align*} \Norm{w_\epsilon\chi_{\{x \in X\,:\,Mf(x)>t\}}}_\pp &\leq \Norm{w\chi_{\{x \in X\,:\,Mf(x)>t\}}}_\pp + \epsilon\Norm{\chi_{\{x \in X\,:\,Mf(x)>t\}}}_\pp.
\intertext{Since $\pp \in \LH$, $M$ satisfies the weak type inequality on $L^\pp(X,\mu)$. This is a result of the sufficiency argument (Section~\ref{sufficiency}) if $p_->1$, and in general it is one case in the main result of \cite{Shukla}. Consequently,}
&\leq C\Norm{fw}_\pp + C\Norm{\epsilon f}_\pp \\ &\leq 2C\Norm{fw_\epsilon}_\pp.
\end{align*}
This shows that $w_\epsilon$ satisfies the weak-type inequality, and does so with a constant depending only on the weak-type inequality constants of $w$ and $1$, both of which are independent of $\epsilon$. From the argument with $\Norm{w^{-1}\chi_B}_\cpp < \infty$, it follows that $w_\epsilon \in \App$. In fact, careful inspection of the previous argument will show that
\[ \Norm{w\chi_B}_\pp\Norm{w_\epsilon^{-1}\chi_B}_\cpp \leq \Norm{w_\epsilon\chi_B}_\pp\Norm{w_\epsilon^{-1}\chi_B}_\cpp \leq K\mu(B) \]
with $K$ depending only on $\pp$ and the weak-type inequality constant (in the $F_\infty$ case the dependency is only on the former, while the $F_0$ case involves $(p')_-$). Since as we said before this is independent of $\epsilon$, we have that $K$ is independent of $\epsilon$. Thus since $w_\epsilon^{-1}$ increases to $w^{-1}$ pointwise, by Lemma~\ref{MCT}, we get that $w \in \App$. While this completes the proof of necessity, it is of note that $w \in \App$ in turn implies that the assumption $\Norm{w^{-1}\chi_B}_\cpp < \infty$ must have been true originally.


\section{Sufficiency} \label{sufficiency}
In this section we prove sufficiency in Theorem \ref{goal}. We first assume that $\mu(X) = \infty$; the finite measure case is much simpler, as we will later show. Consider the following lemma, which is proved in \cite{systems,Kairema}.
\begin{lemma} \label{reduction}
There exists a finite family $\{\mathcal{D}_i\}_{i=1}^N$ of dyadic grids such that
\[ Mf(x) \leq C\sum_{i=1}^N M^{\mathcal{D}_i}f(x) \]
for any function $f$ and almost every $x \in X$.
\end{lemma}
As a result of Lemma $\ref{reduction}$, to prove the boundedness of $M$ it suffices to prove the boundedness of $M^\mathcal{D}$ for an arbitrary dyadic grid $\mathcal{D}$. To that end, fix an exponent $\pp$ with $1<p_- \leq p_+<\infty$, a weight $w \in \App$, and a function $f$; without loss of generality we may assume that $f$ is nonnegative and that $\Norm{fw}_\pp = 1$. It is useful to define the weights $W(\cdot) = w(\cdot)^\pp$ and $\sigma(\cdot) = w(\cdot)^{-\cpp}$, both of which are in $A_\infty$ by Lemma~\ref{Ap to Ainfty} and hence doubling by Lemma~\ref{Ainfty}.

We will want to form the Calder\'on-Zygmund cubes of $f$ (with respect to $\mu$). In order to do so, we must show that $\avgint_{Q_k} |f(x)|\,d\mu \to 0$ as $k \to \infty$ for any nested sequence $\{Q_k\}_{k=1}^\infty$ with $Q_{k-1} \subseteq Q_k \in \mathcal{D}_k$. Fix such a sequence; considering $k=1$, we have as a consequence of $W$ being doubling that
\[ W(Q_1) \leq W(B(x_c(Q_1),C_dd_0)) \leq C_W^{\log_2C_d} W(B(x_c(Q_1),d_0)). \]
By a similar argument, for any $k$,
\[
\frac{1}{W(Q_k)} \leq \frac{C}{W(B(x_c(Q_k),C_dd_0^k))}. \]
Combining these two estimates and applying Lemma \ref{Ainfty}, we get
\[ \frac{W(Q_1)}{W(Q_k)} \leq C\frac{W(B(x_c(Q_1),d_0))}{W(B(x_c(Q_k),C_dd_0^k))} \leq C\left(\frac{\mu(B(x_c(Q_1,d_0))}{\mu(B(x_c(Q_k),C_dd_0^k))}\right)^\delta. \]
If we rearrange and apply the lower mass bound,
\[ W(Q_k) \geq C \mu(B(x_c(Q_k),C_dd_0^k))^\delta \geq C\mu(B(x_c(Q_1),Cd_0^k))^\delta. \]
As $k \to \infty$, by continuity of $\mu$ and the fact that $X = \bigcup_{k=1}^\infty B(x_c(Q_1),Cd_0^k)$, the right side approaches $C\mu(X)^\delta=\infty$, and thus $W(Q_k) \to \infty$. By Lemma~\ref{Holder}, the $\App$ condition, and Lemma~\ref{conditionalbound} respectively, for all $k$ sufficiently large,
\[ \avgint_{Q_k} f(x)\,d\mu \leq C\Norm{fw}_\pp\mu(Q_k)^{-1}\Norm{w^{-1}\chi_{Q_k}}_\cpp \leq C\Norm{w\chi_{Q_k}}^{-1}_\pp \leq CW(Q_k)^{-1/{p_+}}. \]
This gives us the desired limit.

Define $\sigma(x) = w(x)^{-{p'(x)}}$ and decompose $f$ as  $f=f_1+f_2$ where $f_1 = f\chi_{\{f\sigma^{-1} > 1\}}$ and $f_2 = f\chi_{\{f\sigma^{-1} \leq 1\}}$. By sublinearity, $M^\mathcal{D}f \leq M^\mathcal{D}f_1 + M^\mathcal{D}f_2$, and by Lemma \ref{conditionalbound}, for $i=1,2$,
\begin{equation} \int_X |f_i(x)|^{p(x)}w(x)^{p(x)}\,d\mu \leq \Norm{f_iw}_\pp \leq 1. \label{splitnorm} \end{equation}
Hence by Lemma \ref{normish}, to prove the desired inequality it suffices to show that there exists a constant $C$ depending on $X$, $\pp$, and $w$ such that
\begin{equation} \int_X \left(M^\mathcal{D}f_i(x)\right)^{p(x)}w(x)^{p(x)}\,d\mu \leq C, \quad i=1,2. \label{splitbound} \end{equation}

We begin with the estimate of \eqref{splitbound} for $f_1$. Choose $a > C_{CZ}$ and for each $k \in \mathbb{Z}$ let
\[ X_k = \{x \in X\,:\,M^\mathcal{D}f_1(x) > a^{k+1}\}. \]
Since $f \in L^1_{\text{loc}}$ and $\avgint_{Q^k}f(x)\,d\mu \to 0$ as $k \to \infty$, $M^\mathcal{D}f_1$ is finite almost everywhere, and so 
\[ \{x \in X\,:\,Mf_1(x) > 0\} = \bigcup_k X_k \setminus X_{k+1} \]
up to a set of measure zero. Let $\{Q^k_j\}$ be the CZ cubes of $f_1$ at height $a^k$ with respect to $\mu$. Then by Lemma~\ref{CZ}, for all $k$,
\begin{equation} X_k = \bigcup_j Q^k_j. \label{cubedecomp} \end{equation}
Define the sets $E^k_j = Q^k_j \setminus X_{k}$, as in Lemma~\ref{CZ}. Then from \eqref{cubedecomp} we have
\[ X_k \setminus X_{k+1} = \bigcup_j E^k_j. \]
We now estimate:
\begin{align} \int_X M^\mathcal{D}f_1(x)^{p(x)}w(x)^{p(x)}\,d\mu &= \sum_k \int_{X_k \setminus X_{k+1}} M^\mathcal{D}f_1(x)^{p(x)}w(x)^{p(x)}\,d\mu \nonumber \\ &\leq a^{2p_+}\sum_k\int_{X_k \setminus X_{k+1}}a^{kp(x)}w(x)^{p(x)}\,d\mu \nonumber \\ &\leq C\sum_{k,j}\int_{E^k_j}\left(\avgint_{Q^k_j}f_1(y)\,d\mu\right)^{p(x)}w(x)^{p(x)}\,d\mu \nonumber \\ \label{bothbound} &= C\sum_{k,j}\int_{E^k_j}\left(\int_{Q^k_j}f_1(y)\sigma(y)^{-1}\sigma(y)\,d\mu\right)^{p(x)}\mu(Q^k_j)^{-p(x)}w(x)^{p(x)}\,d\mu. \end{align}
Since either $f_1\sigma^{-1} \geq 1$ or $f_1\sigma^{-1} = 0$, 
\begin{multline*} \int_{Q^k_j}f_1(y)\sigma(y)^{-1}\sigma(y)\,d\mu \leq \int_{Q^k_j}(f_1(y)\sigma(y)^{-1})^{p(y)/p_-(Q^k_j)}\sigma(y)\,d\mu \\ \leq \int_{Q^k_j}(f_1(y)\sigma(y)^{-1})^{p(y)}\sigma(y)\,d\mu \leq \int_{Q^k_j}f_1(y)^{p(y)}\,d\mu \leq 1. \end{multline*}
Therefore, 
\begin{multline*} \sum_{k,j}\int_{E^k_j}\left(\int_{Q^k_j}f_1(y)\sigma(y)^{-1}\sigma(y)\,d\mu\right)^{p(x)}\mu(Q^k_j)^{-p(x)}w(x)^{p(x)}\,d\mu \\ \leq \sum_{k,j} \left(\int_{Q^k_j}(f_1(y)\sigma(y)^{-1})^{p(y)/p_-(Q^k_j)}\sigma(y)\,d\mu\right)^{p_-(Q^k_j)}\int_{E^k_j}\mu(Q^k_j)^{-p(x)}w(x)^{p(x)}\,d\mu.
\end{multline*}
If we multiply and divide by $\sigma(Q^k_j)$ and apply H\"{o}lder's inequality with exponent $p_-(Q^k_j)/p_-$, we get
\begin{equation} \leq C\sum_{k,j}\left(\avgint_{Q^k_j}(f_1(y)\sigma(y)^{-1})^{p(y)/p_-}\sigma(y)\,d\mu\right)^{p_-}\int_{E^k_j}\sigma(Q^k_j)^{p_-(Q^k_j)} \mu(Q^k_j)^{-p(x)}w(x)^{p(x)}\,d\mu. \label{prebound} \end{equation}

Assume for the moment that
\begin{equation} \int_{Q^k_j}\sigma(Q^k_j)^{p_-(Q^k_j)}\mu(Q^k_j)^{-p(x)}w(x)^{p(x)}\,d\mu \leq C\sigma(Q^k_j). \label{assumption} \end{equation}
Since $\mu(Q^k_j) \leq C\mu(E^k_j)$ by Lemma~\ref{CZ} and $\sigma \in A_\infty$ by Lemma~\ref{Ap to Ainfty} applied to $w^{-1} \in A_\cpp$, we have from Lemma~\ref{Ainfty} (applied to cubes instead of balls) that $\sigma(Q^k_j) \leq C\sigma(E^k_j)$. Thus \eqref{prebound} is bounded by
\begin{align*}
C\sum_{k,j}\left(\avgint_{Q^k_j}(f_1(y)\sigma(y)^{-1})^{p(y)/p_-}\,d\sigma\right)^{p_-} \sigma(E^k_j) &\leq C\sum_{k,j}\int_{E^k_j} M^\mathcal{D}_\sigma((f_1\sigma^{-1})^{\pp/p_-})(x)^{p_-}\sigma(x)\,d\mu \\ &\leq C\int_X M^\mathcal{D}_\sigma((f_1\sigma^{-1})^{\pp/p_-})(x)^{p_-}\sigma(x)\,d\mu.
\intertext{By Lemma \ref{pbound} and \eqref{splitbound},}
&\leq C\int_X f_1(x)^{p(x)}\sigma(x)^{-p(x)}\sigma(x)\,d\mu \\ &= C\int_X f_1(x)^{p(x)}w(x)^{p(x)}d\mu \\ &\leq C. \end{align*}

We now justify \eqref{assumption}. Observe that the left-hand side is dominated by
\begin{equation} \left(\frac{\sigma(Q^k_j)}{\Norm{w^{-1}\chi_{Q^k_j}}_\cpp}\right)^{p_-(Q^k_j)} \int_{Q^k_j} \Norm{w^{-1}\chi_{Q^k_j}}_\cpp^{p_-(Q^k_j)-p(x)}\Norm{w^{-1}\chi_{Q^k_j}}_\cpp^{p(x)}\mu(Q^k_j)^{-p(x)}w(x)^{p(x)}\,d\mu. \label{tripartite} \end{equation}
We will bound \eqref{tripartite} by showing that, under our hypotheses, it reduces to the $\App$ condition. First, we show that
\begin{equation} \Norm{w^{-1}\chi_{Q^k_j}}_\cpp^{p_-(Q^k_j)-p(x)} \leq C. \label{partone} \end{equation}
If $\Norm{w^{-1}\chi_{Q^k_j}}_\cpp > 1$, then $C=1$ works, so assume otherwise. Then
\begin{multline*}
p(x) - p_-(Q^k_j) = \frac{p'(x)}{p'(x)-1} - \frac{(p')_+(Q^k_j)}{(p')_+(Q^k_j)-1} \\ = \frac{(p')_+(Q^k_j)-p'(x)}{[p'(x)-1][(p')_+(Q^k_j)-1]} \leq \frac{(p')_+(Q^k_j)-(p')_-(Q^k_j)}{[(p')_--1]^2},
\end{multline*}
and so by Lemma \ref{fracexp}, we obtain \eqref{partone}. We would also like to prove the bound
\begin{equation} \left(\frac{\sigma(Q^k_j)}{\Norm{w^{-1}\chi_{Q^k_j}}_\cpp}\right)^{p_-(Q^k_j)} \leq C\sigma(Q^k_j). \label{parttwo} \end{equation}
If $\Norm{w^{-1}\chi_{Q^k_j}}_\cpp > 1$, then by Lemma \ref{conditionalbound},
\[ \left(\frac{\sigma(Q^k_j)}{\Norm{w^{-1}\chi_{Q^k_j}}_\cpp}\right)^{p_-(Q^k_j)} \leq \left(\sigma(Q^k_j)^{1-1/(p')_+(Q^k_j)}\right)^{p_-(Q^k_j)} = \sigma(Q^k_j). \]
If on the other hand $\Norm{w^{-1}\chi_{Q^k_j}}_\cpp \leq 1$, then by Lemma \ref{conditionalbound} (applied twice) and Lemma \ref{fracexp} (applied to cubes),
\begin{align*}
\left(\frac{\sigma(Q^k_j)}{\Norm{w^{-1}\chi_{Q^k_j}}_\cpp}\right)^{p_-(Q^k_j)} &\leq \left(\Norm{w^{-1}\chi_{Q^k_j}}_\cpp^{(p')_-(Q^k_j)-1}\right)^{p_-(Q^k_j)} \\
&\leq \left(\Norm{w^{-1}\chi_{Q^k_j}}_\cpp^{(p')_-(Q^k_j)-1+(p')_+(Q^k_j)-(p')_+(Q^k_j)}\right)^{p_-(Q^k_j)} \\
&\leq C\left(\Norm{w^{-1}\chi_{Q^k_j}}_\cpp^{(p')_+(Q^k_j)-1}\right)^{p_-(Q^k_j)} \\
&\leq C\left(\sigma(Q^k_j)^{\frac{(p')_+(Q^k_j)-1}{(p')_+(Q^k_j)}}\right)^{p_-(Q^k_j)} \\
&\leq C\left(\sigma(Q^k_j)^{\frac{p_-(Q^k_j)'-1}{p_-(Q^k_j)'}}\right)^{p_-(Q^k_j)} \\
&= C\sigma(Q^k_j).
\end{align*}

Applying both \eqref{partone} and \eqref{parttwo} to \eqref{tripartite}, we have that in order to demonstrate \eqref{assumption} it suffices to show
\begin{equation} \int_{Q^k_j}\Norm{w^{-1}\chi_{Q^k_j}}_\cpp^{p(x)}\mu(Q^k_j)^{-p(x)}w(x)^{p(x)}\,d\mu \leq C. \label{partthree} \end{equation}
By Lemma \ref{normish}, this is equivalent to bounding
\[ \Norm{(C\mu(Q^k_j))^{-1}\Norm{w^{-1}\chi_{Q^k_j}}_\cpp w\chi_{Q^k_j}}_\pp = \frac{1}{C\mu(Q^k_j)} \Norm{w\chi_{Q^k_j}}_\pp\Norm{w^{-1}\chi_{Q^k_j}}_\cpp. \]
But by Lemma~\ref{Apcubes} this is, as claimed, the $\App$ condition. Since $w \in \App$, \eqref{assumption} holds for any $k$ and $j$. This completes the proof of \eqref{splitbound} for $f_1$.

\bigskip

We now proceed to show the corresponding bound for $f_2$. Recall that 1, $\sigma$, and $W$ are all in $A_\infty$; from now on, we will use properties of $A_\infty$ without reference.

We would like to fix a particular $\GLH$ base point $x_0$. Let $\{Q^k_j\}$ now represent the CZ cubes of $f_2$ with respect to $\mu$. Choose a nested tower of cubes $\{Q_{k,0}\}$. Since $A_\infty$ weights are doubling, we have that $\mu(Q_{k,0})$, $\sigma(Q_{k,0})$, and $W(Q_{k,0})$ all go to infinity, and as a result we may find a cube $Q_{k_0,0} \equiv Q_0 \in \mathcal{D}_{k_0}$ such that $\mu(Q_0)$, $\sigma(Q_0)$, and $W(Q_0) \geq 1$. By Lemma \ref{basepoint}, we may fix $x_0=x_c(Q_0)$. Let $N_0 = 2A_0C_d$, where $C_d$ is as in Theorem~\ref{dyadic}, and define the sets
\begin{align*} \mathscr{F} &= \{(k,j)\,:\,Q^k_j \subseteq Q_0\} \\ \mathscr{G} &= \{(k,j)\,:\, Q^k_j \not\subseteq Q_0 \text{ and } d(x_0,x_c(Q^k_j)) < N_0d_0^k\} \\ \mathscr{H} &= \{(k,j)\,:\, Q^k_j \not\subseteq Q_0 \text{ and } d(x_0,x_c(Q^k_j)) \geq N_0d_0^k\}. \end{align*}
Observe that $\mathscr{F} \cup \mathscr{G} \cup \mathscr{H} = \mathbb{Z}\times\mathbb{N}$, so that every CZ cube $Q^k_j$ has indices in one of the three sets. By repeating the argument used to obtain \eqref{bothbound} with $f_2$ in place of $f_1$, we may split the corresponding sum into three parts:
\begin{multline*} \int_X M^\mathcal{D}f_2(x)^{p(x)}w(x)^{p(x)}\,d\mu \leq C\sum_{k,j}\int_{E^k_j}\left(\avgint_{Q^k_j}f_2(y)\sigma(y)\sigma(y)^{-1}\,d\mu\right)^{p(x)}w(x)^{p(x)}\,d\mu \\ = C\left(\sum_{(k,j) \in \mathscr{F}} + \sum_{(k,j) \in \mathscr{G}}+\sum_{(k,j) \in \mathscr{H}}\right) = C(I_1 + I_2 + I_3). \end{multline*}
We will bound each of these three sums in turn, beginning with $I_1$. Using that $f_2\sigma^{-1} \leq 1$ to eliminate $f_2$ and then applying \eqref{assumption}, we get
\begin{align*}
I_1 &\leq \sum_{(k,j) \in \mathscr{F}}\int_{E^k_j} \left(\avgint_{Q^k_j}\sigma(y)\,d\mu\right)^{p(x)}w(x)^{p(x)}d\mu \\
&\leq \sum_{(k,j) \in \mathscr{F}}\int_{E^k_j}\sigma(Q^k_j)^{p(x)-p_-(Q^k_j)}\sigma(Q^k_j)^{p_-(Q^k_j)}\mu(Q^k_j)^{-p(x)}w(x)^{p(x)}\,d\mu \\
&\leq \sum_{(k,j) \in \mathscr{F}} (1+\sigma(Q^k_j))^{p_+(Q^k_j)-p_-(Q^k_j)}\int_{E^k_j}\sigma(Q^k_j)^{p_-(Q^k_j)}\mu(Q^k_j)^{-p(x)}w(x)^{p(x)}\,d\mu \\
&\leq C(1+\sigma(Q_0))^{p_+-p_-}\sum_{(k,j)\in\mathscr{F}}\sigma(Q^k_j) \\
&\leq C(1+\sigma(Q_0))^{p_+-p_-}\sum_{(k,j) \in \mathscr{F}}\sigma(E^k_j) \\
&\leq C(1+\sigma(Q_0)))^{p_+-p_-}\sigma(Q_0),
\end{align*}
which is a constant independent of $Q^k_j$ and $f$.

Now to estimate $I_2$, pick $(k,j) \in \mathscr{G}$. Note that if $x_c(Q^k_j) \in Q_0$, then since $Q^k_j \not\subseteq Q_0$ we must have that 
\[ Q_0 \subseteq Q^k_j \subseteq B(x_c(Q^k_j),A_0(C_d+1)N_0d_0^k). \]
On the other hand, if $x_c(Q^k_j) \not\in Q_0 \supseteq B(x_0,d_0^{k_0})$, then by the definition of $\mathscr{G}$,
\[ d_0^{k_0} \leq d(x_0,x_c(Q^k_j)) \leq N_0d_0^k. \]
As a result, since $x_0 \in B(x_c(Q^k_j),N_0d_0^k)$ and $x \in B(x_0,C_dd_0^{k_0})$, for any $x \in Q_0$,
\[ d(x,x_c(Q^k_j)) \leq A_0(d(x,x_0)+d(x_0,x_c(Q^k_j))) \leq A_0(C_dd_0^{k_0}+N_0d_0^k) \leq A_0(C_d+1)N_0d_0^k. \]
It follows that $Q_0 \subseteq B(x_c(Q^k_j),A_0(C_d+1)N_0d_0^k) \equiv B^k_j$ for any $(k,j) \in \mathscr{G}$. Consequently, we have that $W(B^k_j), \sigma(B^k_j) \geq 1$. Note also that by doubling and Lemma \ref{dyadic}, $\mu(Q^k_j) \approx \mu(B^k_j)$. By Lemma~\ref{normish}, we also have that $\Norm{w^{-1}\chi_{Q_0}}_\cpp \geq 1$, and so by Corollary~\ref{cor} (applied to $w^{-1} \in A_\cpp$),
\[ \mu(Q^k_j)^{-1} \leq C\mu(B^k_j)^{-1} \leq C\mu(Q_0)^{-1}\left(\frac{\sigma(Q_0)}{\sigma(B^k_j)}\right)^{1/p'_\infty} \leq C\Norm{w^{-1}\chi_{B^k_j}}^{-1}_\cpp \leq C\Norm{w^{-1}\chi_{Q^k_j}}_\cpp^{-1}. \]
It follows from this inequality and Lemma~\ref{Holder} that
\[ \avgint_{Q^k_j}f_2(y)\,d\mu \leq C\Norm{w^{-1}\chi_{Q^k_j}}^{-1}_\cpp\Norm{f_2w}_\pp\Norm{w^{-1}\chi_{Q^k_j}}_\cpp \leq C. \]
Given this, we may apply Lemma~\ref{DCU2.7} with the exponents $\pp$ and $p_\infty$ to estimate:
\begin{multline} I_2 \leq C\sum_{(k,j) \in \mathscr{G}} \int_{E^k_j}\left(C^{-1}\avgint_{Q^k_j}f_2(y)\,d\mu\right)^{p(x)}w(x)^{p(x)}\,d\mu \\ \leq C_t \sum_{(k,j) \in \mathscr{G}}\int_{E^k_j}\left(\avgint_{Q^k_j}f_2(y)\,d\mu\right)^{p_\infty}w(x)^{p(x)}\,d\mu + \sum_{(k,j) \in \mathscr{G}}\int_{E^k_j}\frac{w(x)^{p(x)}}{(e+d(x_0,x))^{tp_-}}\,d\mu \label{I2} \end{multline}
Arguing as we did in the proof of Lemma~\ref{Ap to Ainfty} to obtain inequality~\eqref{tsum}, we may choose $t$ sufficiently large (depending only on $X$, $Q_0$, $\pp$, and $w$) so that 
\begin{equation} \label{expint} \sum_{(k,j) \in \mathscr{G}}\int_{E^k_j}\frac{w(x)^{p(x)}}{(e+d(x_0,x))^{tp_-}}\,d\mu \leq \int_{X}\frac{w(x)^{p(x)}}{(e+d(x_0,x))^{tp_-}}\,d\mu \leq 1. \end{equation}
We now need only bound the first term of \eqref{I2}. But we have that
\begin{multline*} \sum_{(k,j)\in\mathscr{G}}\int_{E^k_j}\left(\avgint_{Q^k_j}f_2(y)\,d\mu\right)^{p_\infty}w(x)^{p(x)}\,d\mu \\ = \sum_{(k,j)\in\mathscr{G}}\left(\frac{1}{\sigma(Q^k_j)}\int_{Q^k_j}f_2(y)\sigma(y)^{-1}\sigma(y)\,d\mu\right)^{p_\infty}\left(\frac{\sigma(Q^k_j)}{\mu(Q^k_j)}\right)^{p_\infty}W(E^k_j). \end{multline*}
Now invoking \eqref{sumboundop} (applied to $\sigma$ and then $W$, with cubes) as well as the $\App$ condition,
\begin{equation} \sigma(Q^k_j)^{p_\infty-1} = \sigma(Q^k_j)^{p_\infty/p'_\infty} \leq C\Norm{w^{-1}\chi_{Q^k_j}}_\cpp^{p_\infty} \leq C\left(\frac{\mu(Q^k_j)}{\Norm{w\chi_{Q^k_j}}_\pp}\right)^{p_\infty} \leq C\frac{\mu(Q^k_j)^{p_\infty}}{W(Q^k_j)}. \label{reducingestimate} \end{equation}
If we apply this estimate, Lemmas~\ref{pbound} (since by assumption $p_->1$ and we must have $p_\infty \geq p_-$) and \ref{DCU2.7}, and that $\sigma(Q^k_j) \leq C\sigma(E^k_j)$, then we get
\begin{align}
&\sum_{(k,j)\in\mathscr{G}}\left(\frac{1}{\sigma(Q^k_j)}\int_{Q^k_j}f_2(y)\sigma(y)^{-1}\sigma(y)\,d\mu\right)^{p_\infty}\left(\frac{\sigma(Q^k_j)}{\mu(Q^k_j)}\right)^{p_\infty}W(E^k_j). \label{longstart} \\
&\leq C\sum_{(k,j)\in\mathscr{G}}\left(\frac{1}{\sigma(Q^k_j)}\int_{Q^k_j}f_2(y)\sigma(y)^{-1}\sigma(y)\,d\mu\right)^{p_\infty}\sigma(Q^k_j)W(Q^k_j)^{-1}W(E^k_j) \nonumber \\
&\leq C\sum_{(k,j)\in\mathscr{G}}\left(\frac{1}{\sigma(Q^k_j)}\int_{Q^k_j}f_2(y)\sigma(y)^{-1}\sigma(y)\,d\mu\right)^{p_\infty}\sigma(E^k_j) \nonumber \\
&\leq C\sum_{(k,j)\in\mathscr{G}}\int_{E^k_j}M_\sigma(f_2\sigma^{-1})(x)^{p_\infty}\sigma(x)\,d\mu \nonumber \nonumber \\
&\leq C\int_X M_\sigma(f_2\sigma^{-1})(x)^{p_\infty}\sigma(x)\,d\mu \label{prestrongtypeinequality} \\
&\leq C\int_X (f_2(x)\sigma^{-1}(x))^{p_\infty}\sigma(x)\,d\mu \label{poststrongtypeinequality} \\
&\leq C_t\int_X (f_2(x)\sigma(x))^{p(x)}\sigma(x)\,d\mu + \int_X \frac{\sigma(x)}{(e+d(x_0,x))^{tp_-}}\,d\mu \nonumber \\
&\leq C_t\int_Xf_2(x)^{p(x)}w(x)^{p(x)}\,d\mu + \int_X\frac{\sigma(x)}{(e+d(x_0,x))^{tp_-}}\,d\mu. \label{repeatedresult} 
\end{align}
The second term is bounded by a constant independent of $Q^k_j$ and $f$, by an argument identical to that used to prove \eqref{expint} with $\sigma$ in place of $W$. By \eqref{splitnorm}, the first term is also bounded by a constant, and thus $I_2$ is as well.

We now estimate $I_3$. Central to this part of the proof will be that $d(x_0,x)$ is essentially constant on $Q^k_j$; that is,
\begin{equation}\label{quasiconstant} \sup_{x \in Q^k_j} d(x_0,x) \leq R\inf_{x \in Q^k_j} d(x_0,x), \end{equation}
for some constant $R \geq 1$ independent of $k$ and $j$. In fact, we will show that \eqref{quasiconstant} is true with $Q^k_j$ replaced by the ball $A^k_j = N_0^{-1}B^k_j \supseteq Q^k_j$. To that end, fix $(k,j) \in \mathscr{H}$ and choose $x \in A^k_j$. We have that
\begin{multline*} d(x_0,x) \leq A_0[d(x_0,x_c(Q^k_j))+d(x_c(Q^k_j),x)] \\\leq A_0[d(x_0,x_c(Q^k_j))+C_dd_0^k] \leq \left(A_0+\frac{1}{2}\right)d(x_0,x_c(Q^k_j)). \end{multline*}
Conversely,
\begin{multline*} d(x_0,x_c(Q^k_j)) \leq A_0[d(x,x_c(Q^k_j))+d(x_0,x)] \\ = \frac{1}{2}N_0d_0^k + A_0d(x_0,x) \leq \frac{1}{2}d(x_0,x_c(Q^k_j)) + A_0d(x_0,x), \end{multline*}
and so by rearranging terms,
\[ d(x_0,x_c(Q^k_j)) \leq 2A_0d(x_0,x). \]
It follows that $d(x_0,x_c(Q^k_j)) \approx d(x_0,x)$. This is equivalent to $\eqref{quasiconstant}$.

To now estimate $I_3$, we need to divide $\mathscr{H}$ into two subsets,
\[ \mathscr{H}_1 = \{(k,j) \in \mathscr{H}\,:\,\sigma(Q^k_j) \leq 1\}, \quad \mathscr{H}_2 = \{(k,j) \in \mathscr{H}\,:\, \sigma(Q^k_j) > 1\}. \]
We sum first over $\mathscr{H}_1$. Let $x_+ \in \overline{A^k_j}$ be the point which (by continuity of $\pp \in \LLH$) satisfies $p_+(A^k_j) = p(x_+)$. Then by the $\GLH$ condition and \eqref{quasiconstant}, for all $x \in Q^k_j$,
\begin{multline*} |p_+(Q^k_j) - p(x)| \leq |p(x_+)-p_\infty| + |p(x)-p_\infty| \leq \frac{C_\infty}{\log(e+d(x_0,x_+))}+\frac{C_\infty}{\log(e+d(x_0,x))} \\ \leq C_\infty\left[\frac{1}{\log(e+(RA_0)^{-1}d(x_0,x))} + \frac{1}{\log(e+d(x_0,x))}\right] \leq \frac{C_\infty(RA_0+1)}{\log(e+d(x_0,x))}. \end{multline*}
This provides the necessary condition to apply Lemma~\ref{DCU2.8}, from which (bounding the second term with \eqref{expint} as before) we get
\[ \sum_{(k,j) \in \mathscr{H}_1}\int_{E^k_j}\left(\avgint_{Q^k_j}f_2(y)\,d\mu\right)^{p(x)}w(x)^{p(x)}\,d\mu \leq C_t\sum_{(k,j) \in \mathscr{H}_1}\int_{E^k_j}\left(\avgint_{Q^k_j}f_2(y)\,d\mu\right)^{p_+(Q^k_j)} + 1. \]
By appealing to Lemma~\ref{DCU2.9} for the inequality
\[ \mu(Q^k_j)^{p(x)-p_+(Q^k_j)} \leq C, \]
we may bound the sum on the right by
\[ C\sum_{(k,j) \in \mathscr{H}_1}\int_{E^k_j}\left(\frac{1}{\sigma(Q^k_j)}\int_{Q^k_j}f_2(y)\sigma(y)^{-1}\sigma(y)\,d\mu\right)^{p_+(Q^k_j)} \sigma(Q^k_j)^{p_+(Q^k_j)}\mu(Q^k_j)^{-p(x)}w(x)^{p(x)}\,d\mu, \]
and since $f_2\sigma^{-1} \leq 1$, by Lemma~\ref{DCU2.7} we may continue to estimate
\begin{align*}
&\leq C\sum_{(k,j) \in \mathscr{H}_1}\int_{E^k_j}\left(\frac{1}{\sigma(Q^k_j)}\int_{Q^k_j}f_2(y)\sigma(y)^{-1}\sigma(y)\,d\mu\right)^{p_\infty} \sigma(Q^k_j)^{p_+(Q^k_j)}\mu(Q^k_j)^{-p(x)}w(x)^{p(x)}\,d\mu
\\ &\,\,\,\,\,+ C\sum_{(k,j) \in \mathscr{H}_1} \sigma(Q^k_j)^{p_+(Q^k_j)}\mu(Q^k_j)^{-p(x)}\frac{w(x)^{p(x)}}{(e+d(x_0,x))^{tp_-}}\,d\mu
\\ &= CJ_1 + CJ_2.
\end{align*}
To estimate $J_2$ we use that $\sigma(Q^k_j) \leq 1$, then apply \eqref{assumption}---together with the fact that $\sigma(Q^k_j) \leq C\sigma(E^k_j)$, as used in the $f_1$ argument---and subsequently \eqref{quasiconstant}, to get that
\begin{align*}
J_2 &\leq \sum_{(k,j) \in \mathscr{H}_1} \sup_{x \in E^k_j} (e+d(x_0,x))^{-tp_-}\int_{E^k_j}\sigma(Q^k_j)^{p_-(Q^k_j)}\mu(Q^k_j)^{-p(x)}w(x)^{p(x)}\,d\mu \\
&\leq C\sum_{(k,j) \in \mathscr{H}_1} \sup_{x \in E^k_j} (e+d(x_0,x))^{-tp_-}\sigma(E^k_j) \\
&\leq C\sum_{(k,j) \in \mathscr{H}_1} \int_{E^k_j}\frac{\sigma(x)}{(e+d(x_0,x))^{tp_-}}\,d\mu \\
&\leq C\int_X\frac{\sigma(x)}{(e+d(x_0,x))^{tp_-}}\,d\mu,
\end{align*}
which is the same quantity as the second term in \eqref{repeatedresult}, which we argued was bounded by a constant at the end of the estimate for $I_2$.

Similarly, to estimate $J_1$ we may use that $\sigma(Q^k_j) \leq 1$ and \eqref{assumption} to get that
\begin{align*} J_1 &\leq C\sum_{(k,j) \in \mathscr{H}_1} \left(\frac{1}{\sigma(Q^k_j)}\int_{Q^k_j}f_2(y)\sigma(y)^{-1}\sigma(y)\,d\mu\right)^{p_\infty} \sigma(Q^k_j).
\intertext{Again using that $\sigma(Q^k_j) \leq C\sigma(E^k_j)$, we get that}
&\leq C\sum_{(k,j) \in \mathscr{H}_1} \left(\frac{1}{\sigma(Q^k_j)}\int_{Q^k_j}f_2(y)\sigma(y)^{-1}\sigma(y)\,d\mu\right)^{p_\infty} \sigma(E^k_j) \\
&\leq C\int_X M_\sigma(f_2\sigma^{-1})(x)^{p_\infty}\sigma(x)\,d\mu.
\end{align*}
But this is yet another quantity that appears near the end of the $I_2$ estimate, and thus it is bounded by a constant. This completes the estimate for $\mathscr{H}_1$.

Finally, we now estimate the sum over $\mathscr{H}_2$. By Lemma~\ref{Holder},
\[ \int_{Q^k_j} f_2(y)\,d\mu \leq c\Norm{f_2w}_\pp\Norm{w^{-1}\chi_{Q^k_j}}_\cpp \leq c\Norm{w^{-1}\chi_{Q^k_j}}_\cpp. \]
Thus we can apply Lemma~\ref{DCU2.7} to get
\begin{align*}
&\sum_{(k,j) \in \mathscr{H}_2} \int_{E^k_j}\left(\avgint_{Q^k_j}f_2(y)\,d\mu\right)^{p(x)}w(x)^{p(x)}\,d\mu \\
&\,\,\,\,\,\,\leq C\sum_{(k,j) \in \mathscr{H}_2} \int_{E^k_j}\left(c\Norm{w^{-1}\chi_{Q^k_j}}^{-1}_\cpp\int_{Q^k_j}f_2(y)\,d\mu\right)^{p(x)} \left(\frac{\Norm{w^{-1}\chi_{Q^k_j}}_\cpp}{\mu(Q^k_j)}\right)^{p(x)}w(x)^{p(x)}\,d\mu \\
&\,\,\,\,\,\,\leq C\sum_{(k,j) \in \mathscr{H}_2} \int_{E^k_j}\left(\Norm{w\chi_{Q^k_j}}^{-1}_\cpp\int_{Q^k_j}f_2(y)\,d\mu\right)^{p_\infty}\left(\frac{\Norm{w^{-1}\chi_{Q^k_j}}_\cpp}{\mu(Q^k_j)}\right)^{p(x)}w(x)^{p(x)}\,d\mu \\
&\,\,\,\,\,\,\,\,\,\,\,+\sum_{(k,j) \in \mathscr{H}_2} \int_{E^k_j}\left(\frac{\Norm{w^{-1}\chi_{Q^k_j}}_\cpp}{\mu(Q^k_j)}\right)^{p(x)}\frac{w(x)^{p(x)}}{(e+d(x_0,x))^{tp_-}}\,d\mu \\ &\,\,\,\,\,\,= K_1 + K_2.
\end{align*}
To estimate $K_2$, note that $1 \leq \sigma(Q^k_j) \leq C\sigma(E^k_j)$, so $\sigma(E^k_j) > \epsilon$ for some fixed constant $\epsilon > 0$. Therefore, by \eqref{partthree} and \eqref{quasiconstant} we have that
\begin{align*}
K_2 &\leq \epsilon^{-1}\sum_{(k,j) \in \mathscr{H}_2} \sup_{x \in E^k_j}(e+d(x_0,x))^{-tp_-} \epsilon \int_{Q^k_j}\left(\frac{\Norm{w^{-1}\chi_{Q^k_j}}_\cpp}{\mu(Q^k_j)}\right)^{p(x)}w(x)^{p(x)}\,d\mu \\
&\leq C\sum_{(k,j) \in \mathscr{H}_2} \sup_{x \in E^k_j}(e+d(x_0,x))^{-tp_-}\sigma(E^k_j) \\
&\leq C\int_X \frac{\sigma(x)}{(e+d(x_0,x))^{tp_-}}\,d\mu,
\end{align*}
which as we argued in $J_2$ and $I_2$ is bounded by a constant.

To estimate $K_1$, we use \eqref{sumboundop} to get
\[ \Norm{w^{-1}\chi_{Q^k_j}}_\cpp^{-p_\infty}\sigma(Q^k)j)^{p_\infty} \leq C\sigma(Q^k_j)^{p_\infty-p_\infty/p'_\infty} = C\sigma(Q^k_j). \]
Therefore, applying \eqref{partthree} and that $\sigma(Q^k_j) \leq C\sigma(E^k_j)$, we have
\begin{align*}
K_1 &= \sum_{(k,j) \in \mathscr{H}_2}\int_{E^k_j}\left(\frac{1}{\sigma(Q^k_j)}\int_{Q^k_j}f_2(y)\,d\mu\right)^{p_\infty}\Norm{w^{-1}\chi_{Q^k_j}}_\cpp^{p(x)-p_\infty}\frac{\sigma(Q^k_j)^{p_\infty}}{\mu(Q^k_j)^{p(x)}}w(x)^{p(x)}\,d\mu \\
&\leq C\sum_{(k,j) \in \mathscr{H}_2}\left(\frac{1}{\sigma(Q^k_j)}\int_{Q^k_j}f_2(y)\,d\mu\right)^{p_\infty}\sigma(Q^k_j)\int_{Q^k_j}\Norm{w^{-1}\chi_{Q^k_j}}_\cpp^{p(x)}\mu(Q^k_j)^{-p(x)}w(x)^{p(x)}\,d\mu \\
&\leq C\sum_{(k,j) \in \mathscr{H}_2}\left(\frac{1}{\sigma(Q^k_j)}\int_{Q^k_j}f_2(y)\,d\mu\right)^{p_\infty}\sigma(E^k_j) \\
&\leq C\int_X M_\sigma(f_2\sigma^{-1})(x)^{p_\infty}\sigma(x)\,d\mu.
\end{align*}
This last term is the same quantity that appeared in \eqref{prestrongtypeinequality}, which as we argued in the estimates for $J_2$ and $I_2$ is bounded by a constant. This completes the estimate for $I_3$, and thus gives us the desired estimate for $f_2$, completing our proof of the sufficiency of the $\App$ condition in Theorem~\ref{goal} for the strong-type inequality.

\bigskip

\subsection*{The finite case}
If $\mu(X) < \infty$, we may apply the same proof as in the infinite case, with some modifications. For each $i=1,2$, in accordance with Lemma~\ref{CZ}, we may only construct the CZ cubes at heights greater than $\lambda_0 = \avgint_X f_i\,d\mu$. Note that with the assumption that $\Norm{fw}_\pp = 1$ as before, we have from Lemma~\ref{Holder} that
\[ \lambda_0 \leq 4\mu(X)^{-1}\Norm{f_iw}_\pp\Norm{w^{-1}}_\cpp \leq 4\mu(X)^{-1}\Norm{w^{-1}}_\cpp. \]
By Lemma~\ref{normish}, this is bounded by a constant, since from the $\App$ condition with $B=X$,
\[ \Norm{w^{-1}}_\cpp \leq C\mu(X)\Norm{w}_\pp^{-1}. \]
Fix $a = 2C_{CZ}$ and let $Q^k_j$ denote, as before, the CZ cubes of $f_i$ at height $a^k$, where $k \geq k_0 = \floor{\log_a\lambda_0+1}$. These cubes cover only $X_{k_0} = \{x \in X\,:\,M^\mathcal{D}f_i(x) > \lambda_0\}$. If, however, we define
\[ X_0 = \{M^\mathcal{D}f_i(x) \leq \lambda_0\} = X \setminus X_{k_0} \]
then
\[ X = \left(\bigcup_{k=k_0}^\infty X_k \setminus X_{k+1}\right)\bigcup X_0. \]
Thus the analogous argument to \eqref{bothbound} proceeds as
\begin{align*}
\int_X M^\mathcal{D}f_i(x)^{p(x)}&w(x)^{p(x)}\,d\mu\\
&= \int_{X_0} M^\mathcal{D}f_i(x)^{p(x)}w(x)^{p(x)}\,d\mu + \sum_{k=k_0}^\infty \int_{X_k \setminus X_{k+1}} M^\mathcal{D}f_i(x)^{p(x)}w(x)^{p(x)}\,d\mu \\ &\leq \lambda_0W(X) + C\sum_{k \geq k_0, j} \int_{E^k_j}\left(\int_{Q^k_j}f_i\sigma^{-1}\sigma\,d\mu\right)^{p(x)}\mu(Q^k_j)^{-p(x)}w(x)^{p(x)}\,d\mu.
\end{align*}
Since $\lambda_0$ is bounded by a constant, the first term depends only on $X$, $\mathcal{D}$, and $\pp$. For $f_1$, the second term may be controlled by an argument identical to that of the infinite case.

The $f_2$ case, on the other hand, simplifies greatly: essentially, we just choose $Q_0 = X$, and so $I_2 = I_3 = 0$. More explicitly, since $f_2\sigma^{-1} \leq 1$, if $\sigma(X) \geq 1$, then by \eqref{assumption} and the fact that $\sigma(Q^k_j) \leq C\sigma(E^k_j)$, the second term in the above expression is bounded by
\begin{align*}
\sum_{k \geq k_0, j} \int_{E^k_j}\sigma(Q^k_j)^{p(x)}&\mu(Q^k_j)^{-p(x)}w(x)^{p(x)}\,d\mu \\ &\leq \sum_{k \geq k_0, j} \int_{E^k_j}\sigma(X)^{p(x)}\left(\frac{\sigma(Q^k_j)}{\sigma(X)}\right)^{p(x)}\mu(Q^k_j)^{-p(x)}w(x)^{p(x)}\,d\mu \\
&\leq \sigma(X)^{p_+}\sum_{k \geq k_0,j} \int_{E^k_j} \left(\frac{\sigma(Q^k_j)}{\sigma(X)}\right)^{p_-(Q^k_j)}\mu(Q^k_j)^{-p(x)}w(x)^{p(x)}\,d\mu \\
&\leq \sigma(X)^{p_+-p_-}\sum_{k \geq k_0,j} C\sigma(E^k_j) \\
&\leq C\sigma(X)^{p_+-p_-+1}.
\end{align*}
If $\sigma(X) < 1$, simply exchange $p_+$ with $p_-$. This proves sufficiency for $\mu(X) < \infty$.

\newpage
\nocite{*}

\bibliographystyle{plain}
\bibliography{bibliography}

\end{document}